\documentclass[final]{siamart0216}
\usepackage{amsfonts}
\usepackage{todonotes}
\usepackage{amsfonts,amsmath,amssymb}
\usepackage{mathrsfs,mathtools,stmaryrd,wasysym}
\usepackage{enumerate}
\usepackage{xspace,mydef}
\usepackage{esint}
\usepackage{graphicx,psfrag}
\usepackage{hyperref}
\usepackage{pgf,tikz,pgfplots}
\usepackage{pstricks-add}
\usepackage{enumitem}
\usepackage{yfonts}

\usetikzlibrary{arrows}
\usepackage{rotating}
\DeclareMathAlphabet{\mathpzc}{OT1}{pzc}{m}{it}

\newsiamremark{remark}{Remark}

\newcommand{\norm}[1]{{\left\vert\kern-0.25ex\left\vert\kern-0.25ex\left\vert #1 
\right\vert\kern-0.25ex\right\vert\kern-0.25ex\right\vert}}

\newcommand{\TheTitle}{Finite element approximation for a convective Brinkman--Forchheimer problem coupled with a heat equation}
\newcommand{\ShortTitle}{A Convective Brinkman--Forchheimer model and a heat equation}
\newcommand{\TheAuthors}{G.~Campa\~na, P. Mu\~noz and E.~Ot\'arola}

\headers{\ShortTitle}{\TheAuthors}

\title{{\TheTitle}\thanks{GC is partially supported by ANID through Subdirecci\'on de Capital Humano/Doctorado Nacional/2020--21200920. EO is partially supported by ANID through FONDECYT project 1220156. PM is supported by UTFSM through Programa de Incentivo a la Investigaci\'on Cient\'ifica (PIIC).}}

\author{Gilberto Campa\~na\thanks{Departamento de Ciencias, Universidad T\'ecnica Federico Santa Mar\'ia, Valpara\'iso, Chile. (\email{gilberto.campana@usm.cl}).}
\and
Pablo Mu\~noz\thanks{Departamento de Matem\'atica, Universidad T\'ecnica Federico Santa Mar\'ia, Valpara\'iso, Chile. (\email{pablo.munozssu.14@sansano.usm.cl}).}
\and
Enrique Ot\'arola\thanks{Departamento de Matem\'atica, Universidad T\'ecnica Federico Santa Mar\'ia, Valpara\'iso, Chile. (\email{enrique.otarola@usm.cl}, \url{http://eotarola.mat.utfsm.cl/}).}
}

\ifpdf
\hypersetup{
  pdftitle={\TheTitle},
  pdfauthor={\TheAuthors}
}
\fi

\date{Draft version of \today.}

\begin{document}

\maketitle

\begin{abstract}
We investigate a convective Brinkman--Forchheimer problem coupled with a heat equation. The investigated model considers thermal diffusion and viscosity depending on the temperature. We prove the existence of a solution without restriction on the data and uniqueness when the solution is slightly smoother and the data is suitably restricted. We also propose a finite element discretization scheme for the considered model and derive convergence results and a priori error estimates. Finally, we illustrate the theory with numerical examples.
\end{abstract}




\section{Introduction}
\label{sec:intro}
Let $\Omega\subset\mathbb{R}^d$, with $d\in\{2,3\}$, be an open and bounded domain with Lipschitz boundary $\partial\Omega$. In this work, we continue our research program focused on the analysis and approximation of fluid flow models coupled with a heat equation. In particular, we are now interested in developing finite element methods for the temperature distribution of a fluid modeled by a convection--diffusion equation coupled with the convective Brinkman--Forchheimer equations. This model can be described by the following \emph{nonlinear} system of partial differential equations (PDEs):
\begin{equation}\label{eq:model}
	\left\{
	\begin{aligned}
		-\text{div}(\nu(T)\nabla\mathbf{u}) + (\mathbf{u}\cdot\nabla)\mathbf{u}+ \mathbf{u} +|\mathbf{u}|^{s-2}\mathbf{u}+ \nabla \mathsf{p}  &=  \mathbf{f} \text{ in }\Omega,
		\quad
		\text{div }\mathbf{u}                                                                                                                 =   0 \text{ in }\Omega, \\
		-\text{div}(\kappa(T)\nabla T)+\mathbf{u}\cdot\nabla T                                                                                &=  \textnormal{g} \text{ in }\Omega,
	\end{aligned}
	\right.
\end{equation}
supplemented with the Dirichlet boundary conditions $\mathbf{u}=\mathbf{0}$ and $T=0$ on $\partial \Omega$. The data of the model are the external density force $\mathbf{f}$, the external heat source $\textnormal{g}$, the viscosity coefficient $\nu$, and the thermal diffusivity coefficient $\kappa$. We note that $\nu(\cdot)$ and $\kappa(\cdot)$ are coefficients that can depend nonlinearly on the temperature. The parameter $s$ is chosen so that  $s\in[3,4]$. The unknowns of the system are the velocity field $\mathbf{u}$, the pressure $\mathsf{p}$, and the temperature $T$ of the fluid.

The analysis and discretization of various incompressible, non-isothermal flows have become increasingly important for a variety of research areas in science and technology. This is due to the numerous applications in industry, e.g., in the design of heat exchangers and chemical reactors, cooling processes, and polymer processing, to name but a few. For advances in the numerical approximation of non-isothermal flows governed by the Navier Stokes equations and a suitable temperature equation, as well as for the so-called Boussinesq problem and its generalizations, we refer the interested reader to \cite{MR3342221,MR2111747,MR4053090,MR4265062,MR1364404,MR1051838,MR2802085,MR1681112,MR3232446,MR2135787}. In this context, we also mention the works \cite{ACFO:22,MR3802674,MR3523581,MR4036533,MR4041519,GMRB} for similar results when the Navier--Stokes equations are replaced by the Darcy equations. Finally, we refer the reader to \cite{MR4659334,MR4448201,MR4337455,MR4684228} for the analysis of finite element schemes for the coupling of a heat equation with the so-called Darcy--Forchheimer equations.

Now that we have briefly explained the importance of non-isothermal flows and mentioned some applications, we would like to turn our attention to the so-called convective Brinkman--Forchheimer model. In the following, we will divide the discussion into two parts:

$\bullet$ The Darcy--Forchheimer equations: As stated in \cite{MR2425154}, Darcy's law, namely $\mathbf{u} = - K \nabla \mathsf{p}/\mu$, is a linear relation that describes the creep flow of Newtonian fluids in porous media. It is supported by years of experimental data and has numerous applications in engineering. However, based on flow experiments in sand packs, Forchheimer realized that Darcy's law was not adequate for moderate Reynolds numbers. He found that the relationship between the pressure gradient and the Darcy
velocity was nonlinear and that this nonlinearity appeared to be quadratic for a variety of experimental data \cite{forchheimer1901wasserbewegung}; see \cite[page 162]{MR2425154} and \cite[Section 6, page 55]{whitaker1996forchheimer}. This modification of the Darcy equations is usually referred to as the Darcy--Forchheimer equations and has been studied in detail in several papers. We refer the reader to the following non-exhaustive list of references  \cite{MR4092292,MR4049400,MR2425154,MR2948707,MR3022234}. As mentioned in \cite{Beyond_Beta}, Forchheimer also noted in his 1901 publication that some data sets could not be described by the quadratic correction. Therefore, he suggested adding a cubic term to describe the behavior of the observed flow in these cases. He also postulated that the Darcy's law correction could allow a polynomial expression in $\mathbf{u}$, e.g., $\mathbf{u} + |\mathbf{u}|\mathbf{u} + |\mathbf{u}|^{2}\mathbf{u}$ and $\mathbf{u} + |\mathbf{u}|^{s-2}\mathbf{u}$; see \cite[page 59]{muskat1937flow}, \cite[\S 2.3]{firdaouss1997nonlinear}, and \cite[page 12]{straughan2008stability}. In practice, the exponent $s$ takes the value $3$ and $4$ in several applications; see \cite{balhoff2010polynomial,firdaouss1997nonlinear,MR2027361,MR1090724,MR1856630,rojas1998nonlinear,skjetne1999new,MR3636305} for the case $s = 4$. As mentioned in \cite[page 133]{skjetne1999new}, the use of the fractional value $s = 7/2$ is also considered in the literature. For these reasons, we consider the parameter $s \in [3,4]$ in our work.

$\bullet$ The convective Brinkman--Forchheimer equations: The inclusion of $-\Delta \mathbf{u}$ and the convective term $(\mathbf{u} \cdot \nabla) \mathbf{u}$ in the Darcy--Forchheimer equations leads to the so-called convective Brinkman--Forchheimer model \eqref{eq:model}. Assuming a two-dimensional stationary, isotropic, incompressible, homogeneous flow through a fluid-saturated porous medium, this system was derived by the authors of \cite{VAFAI1981195} as the governing momentum equation based on local volume averaging and matched asymptotic expansion; see also \cite{lastone}. Further justifications for the inclusion of the so-called Brinkman and convective terms in \eqref{eq:model} were later presented in \cite{VAFAI199511}. We refer the interested reader to \cite{ACO:24:JSC,MR4328398,guo2005lattice,MR3967591,MR3636305} for further insights, analysis, and applications of this model.

Our problem \eqref{eq:model} is related to the coupling of the Brinkman--Forchheimer equations and the so-called double-diffusion equations. There are several articles in the literature on the analysis and discretization of this coupling; see, e.g., \cite{MR1418526,MR3820815,MR4171897,MR4344862,MR4659442}. With the exception of the recently published work \cite{MR4659442}, all the previously cited papers consider the case $s=3$. On the other hand, the models considered in \cite{MR4171897,MR4344862,MR4659442} couple the concentration variable in the forcing term of the momentum equations of the Brinkman--Forchheimer system as $\mathbf{f}(\boldsymbol{\phi})$ for physical reasons other than ours. In contrast, in our work we follow the approach considered in \cite{MR4673339,MR3802674,MR4041519,MR3267163}, which takes up the physical considerations presented in \cite{Hooman,MR3168603}, and considers viscosity and thermal diffusion coefficients, which may depend nonlinearly on the temperature $T$.

As far as we know, this is the first paper dealing with the analysis and numerical approximation of the nonlinear coupled problem \eqref{eq:model} in its present form.
Since the problem involves several sources of nonlinearity, the analysis and discretization are anything but trivial. In the following, we list what we consider to be the most important contributions of our work:

$\bullet$ \emph{Existence and uniqueness of solutions:} We introduce a concept of weak solution for the nonlinear problem \eqref{eq:model} and show the existence of solutions without restriction on the data; the latter is derived using Galerkin's method and Brouwer's fixed point theorem. Moreover, we obtain a global uniqueness result when the solution is slightly smoother, and the data is suitably restricted.

$\bullet$ \emph{Finite element approximation:} We propose a finite element discretization for the system \eqref{eq:model} based on the following two classical inf-sup stable pairs: the \emph{Taylor--Hood element} and the \emph{mini element}. We approximate the temperature variable of the fluid with continuous piecewise linear/quadratic finite elements. As in the continuous case, we show the existence of solutions without restriction on the data and global uniqueness when the solution is slightly smoother, and the data is suitably restricted. We also show the existence of a subsequence of discrete solutions that converge to a solution of the continuous problem \eqref{eq:model}. \emph{This result holds without assumptions on the data and solutions beyond what is required to obtain well-posedness.}

$\bullet$ \emph{A priori error estimates:} We derive a quasi-best approximation result for the proposed numerical method. In doing so, we assume suitable smallness conditions and regularity assumptions for the solution. Under the moderate assumption that $(\mathbf{u},\mathsf{p},T) \in \mathbf{H}^{3}(\Omega) \times H^2(\Omega) \times H^{3}(\Omega)$, we obtain optimal error estimates in a standard energy norm for the approximation with the Taylor--Hood element. For the approximation with the mini element, we obtain optimal error estimates in such a standard energy norm assuming that $(\mathbf{u},\mathsf{p},T) \in \mathbf{H}^{2}(\Omega) \times H^1(\Omega) \times H^{2}(\Omega)$.

$\bullet$ \emph{Numerical simulations:} We computationally investigate the effects of the Forchheimer exponent $s$ in the so-called heated lid-driven cavity flow problem and show how this parameter affects the velocity and pressure of the fluid and, in particular, the position of the counter-rotating vortices that occur in the cavity.

The structure of our manuscript is outlined below. In section \ref{sec:notation}, we establish the notation and the preliminary material. In section \ref{sec:coupled_problem}, we present a weak formulation for the system \eqref{eq:model} and show the existence of solutions without restrictions on the data and uniqueness under suitable smallness conditions on the data. In section \ref{sec:fem}, we develop a finite element scheme, investigate its convergence properties, and derive a priori error estimates to control the error in the approximation of the velocity, pressure, and temperature variables. We conclude in section \ref{sec:numericalexample} with a series of numerical experiments that illustrate and go beyond the theory.


\section{Notation and preliminary remarks}
\label{sec:notation}
We begin this section by establishing the notation and the framework within which we will work.


\subsection{Notation}
We use the standard notation for Lebesgue and Sobolev spaces. The spaces of vector-valued functions and the vector-valued functions themselves are denoted by bold letters. In particular, we set
$
\mathbf{V}(\Omega):=\{\mathbf{v}\in\mathbf{H}_0^1(\Omega):\text{div }\mathbf{v}=0\}.
$

If $\mathcal{X}$ and $\mathcal{Z}$ are Banach function spaces, we write $\mathcal{X} \hookrightarrow \mathcal{Z}$ to denote that $\mathcal{X}$ is continuously embedded in $\mathcal{Z}$. We denote by $\mathcal{Z}'$ and $\|\cdot\|_{\mathcal{Z}}$ the dual and the norm of $\mathcal{Z}$, respectively. We denote by $\langle \cdot, \cdot \rangle_{\mathcal{Z}',\mathcal{Z}'}$ the duality paring between $\mathcal{Z}'$ and $\mathcal{Z}$;  if the underlying spaces are clear from the context, we simply write $\langle \cdot, \cdot \rangle$. Given $\mathfrak{p} \in (1,\infty)$, we denote by $\mathfrak{p}\prime \in (1,\infty)$ its H\"older conjugate, i.e., $\mathfrak{p}\prime$ is such that $1/\mathfrak{p} + 1/\mathfrak{p}\prime= 1$. The relation $a\lesssim b$ means that $a\leq Cb$, where $C$ is a constant that does not depend on $a$, $b$, or the discretization parameters. The value of $C$ can change each time it occurs. If the specific value of a constant is important, we give it a name. For example, in our work we use $\mathcal{C}_{4 \hookrightarrow 2}$ to denote the
best constant in the embedding $H_0^1(\Omega) \hookrightarrow L^4(\Omega)$. To simplify the notation, we also use $\mathcal{C}_{4 \hookrightarrow 2}$ in the vectorial case $\mathbf{H}_0^1(\Omega) \hookrightarrow \mathbf{L}^4(\Omega)$.


\subsection{A convective Brinkman--Forchheimer problem}
\label{sec:Brinkman_Darcy_Forchheimer}
In this section, we investigate existence and uniqueness results for the following weak formulation of a convective Brinkman--Forchheimer problem: Given an external density force $\mathbf{f} \in \mathbf{H}^{-1}(\Omega)$, find a velocity-pressure pair $(\mathbf{u},\mathsf{p})\in \mathbf{H}_0^1(\Omega)\times L_0^2(\Omega)$ such that
\begin{equation}
\label{eq:brinkman_darcy_forchheimer_problem}
\begin{array}{rcl}
\displaystyle
\int_\Omega
\left(
\nu\nabla\mathbf{u}\cdot\nabla\mathbf{v}
+
(\mathbf{u}\cdot\nabla)\mathbf{u}\cdot\mathbf{v}
+
\mathbf{u}\cdot\mathbf{v}
+
|\mathbf{u}|^{s-2}\mathbf{u}\cdot\mathbf{v}
-
\mathsf{p}\text{div }\mathbf{v}
\right)
\mathrm{d}x
&
=
&
\langle\mathbf{f}, \mathbf{v}\rangle,
\\
\displaystyle
\int_\Omega \mathsf{q}\text{div }\mathbf{u}\mathrm{d}x
&
=
&
0,
\end{array}
\end{equation}
for all $\mathbf{v}\in \mathbf{H}_0^1(\Omega)$ and $\mathsf{q}\in {L}_0^2(\Omega)$,  respectively. The function $\nu$ belongs to $C^{0,1}(\mathbb{R})$, and it is assumed to be strictly positive and bounded. Namely, we assume that there are positive constants $\nu_{-}$ and $\nu_{+}$ such that $0<\nu_{-} \leq \nu(r)\leq \nu_{+}$ for every $ r \in \mathbb{R}$. We denote by $\mathcal{L}_{\nu}$ the Lipschitz constant of the function $\nu$.

With the exception of $\mathrm{I}_s:= \int_{\Omega} |\mathbf{u}|^{s-2}\mathbf{u}\cdot\mathbf{v} \mathrm{d}x$, all terms in system \eqref{eq:brinkman_darcy_forchheimer_problem} are trivially well-defined in our space setting. A bound for $\mathrm{I}_s$ can be derived as follows:
\begin{equation}
 \left| \mathrm{I}_s \right|
 \leq
 \| \mathbf{u} \|_{\mathbf{L}^{\tau(s-2)}(\Omega)}^{s-2} \| \mathbf{u} \|_{\mathbf{L}^{\mu}(\Omega)} \| \mathbf{v} \|_{\mathbf{L}^{\mu}(\Omega)},
 \qquad
\tau^{-1} + 2\mu^{-1} = 1,
\label{eq:estimate_for_Forchheimer_1}
\end{equation}
where we have used H\"older's inequality. We now apply the standard Sobolev embeddings from \cite[Theorem 4.12, Cases \textbf{B} and \textbf{C}]{MR2424078} to conclude that $\mathbf{H}_0^1(\Omega) \hookrightarrow \mathbf{L}^{\iota}(\Omega)$ for $\iota < \infty$ when $d=2$ and $\iota \leq 6$ when $d=3$. We can, therefore, set $\tau = q$ for some $q >1$ in two dimensions and $\tau = 3/2$ in three dimensions. It follows that $s- 2 \in [1,2]$ and therefore that $\tau(s-2) \leq 2\tau \leq 3$. From this, we can derive the estimate
\begin{equation}
 \left| \mathrm{I}_s
 \right|
 \lesssim
 \| \nabla \mathbf{u} \|_{\mathbf{L}^{2}(\Omega)}^{s-1} \| \nabla \mathbf{v} \|_{\mathbf{L}^{2}(\Omega)},
\label{eq:estimate_for_Forchheimer_2}
\end{equation}
by again using the Sobolev embedding $\mathbf{H}_0^1(\Omega) \hookrightarrow \mathbf{L}^{\iota}(\Omega)$.

\begin{remark}[boundedness of $\mathrm{I}_s$]
Let us first note that the bound \eqref{eq:estimate_for_Forchheimer_2} is not the only way to control the term $\mathrm{I}_s$. Our bound exploits the maximum $L^{\iota}(\Omega)$-regularity of a function in $H_0^1(\Omega)$ and implies that the estimate \eqref{eq:estimate_for_Forchheimer_2} is also valid for larger values of the parameter $s$. Nevertheless, in our work, we consider $s \in [3,4]$ due to the physical considerations discussed in the introduction.
\end{remark}

\subsubsection{Existence}
To present existence and uniqueness results for system \eqref{eq:brinkman_darcy_forchheimer_problem}, we introduce the bilinear forms $a_{L}: [\mathbf{H}_0^1(\Omega)]^2 \to\mathbb{R}$ and $b: \mathbf{H}_0^1(\Omega)\times L_0^2(\Omega)\to\mathbb{R}$ by
\begin{align*}\label{eq:form_LINEAL}
a_{L}(\mathbf{u},\mathbf{v})
:=
\int_{\Omega} (\nu \nabla\mathbf{u}\cdot\nabla\mathbf{v}
+
\mathbf{u}\cdot\mathbf{v}) \mathrm{d}x,
\qquad b(\mathbf{v},\mathsf{q}):=-\int_{\Omega}  \mathsf{q}\,\text{div}\,\mathbf{v}\mathrm{d}x,
\end{align*}
respectively. We also introduce forms associated with the nonlinear terms $(\mathbf{u}\cdot\nabla)\mathbf{u}$ and $|\mathbf{u}|^{s-2}\mathbf{u}$ in \eqref{eq:brinkman_darcy_forchheimer_problem}. Namely, we define
$a_{N}: [\mathbf{H}_0^1(\Omega)]^3 \to \mathbb{R}$
and $a_{F}: [\mathbf{H}_0^1(\Omega)]^3 \to\mathbb{R}$ by
\begin{align*}\label{eq:form_NOLINEAL}
a_{N}(\mathbf{u};\mathbf{w},\mathbf{v})
:=
\int_{\Omega}(\mathbf{u}\cdot\nabla)\mathbf{w}\cdot\mathbf{v}\mathrm{d}x,
\qquad
a_{F}(\mathbf{u};\mathbf{w},\mathbf{v})
:=
\int_{\Omega} |\mathbf{u}|^{s-2}\mathbf{w}\cdot\mathbf{v}\mathrm{d}x,
\end{align*}
respectively. The form $a_{N}$ satisfies the following properties \cite[Chapter II, Lemma 1.3]{MR0603444}, \cite[Chapter IV, Lemma 2.2]{MR851383}: Let $\mathbf{u}\in \mathbf{V}(\Omega)$ and $\mathbf{v},\mathbf{w}\in \mathbf{H}_0^1(\Omega)$. Then, we have
\begin{align}\label{eq:properties_navier}
a_{N}(\mathbf{u};\mathbf{v},\mathbf{w})+a_{N}(\mathbf{u};\mathbf{w},\mathbf{v})=0,\qquad a_{N}(\mathbf{u};\mathbf{v},\mathbf{v})=0.
\end{align}
Moreover, $a_{N}$ is well-defined and continuous on $[\mathbf{H}_0^1(\Omega)]^3$: There is $\mathcal{C}_{N}>0$ such that
\begin{align}\label{eq:estimates_navier}
|a_{N}(\mathbf{u};\mathbf{v},\mathbf{w})|\leq \mathcal{C}_{N}\|\nabla \mathbf{u}\|_{\mathbf{L}^2(\Omega)}\|\nabla \mathbf{v}\|_{\mathbf{L}^2(\Omega)} \|\nabla \mathbf{w}\|_{\mathbf{L}^2(\Omega)};
\end{align}
see \cite[Chapter II, Lemma 1.1]{MR0603444} and \cite[Lemma IX.1.1]{MR2808162}. As a final ingredient, we introduce $a: [\mathbf{H}_0^1(\Omega)]^3 \to\mathbb{R}$ by $a(\mathbf{u}; \mathbf{u}, \mathbf{v}):=a_{L}(\mathbf{u},\mathbf{v})+a_{N}(\mathbf{u};\mathbf{u},\mathbf{v})+a_{F}(\mathbf{u};\mathbf{u},\mathbf{v})$.

Having introduced all these ingredients, we can rewrite the weak formulation \eqref{eq:brinkman_darcy_forchheimer_problem} as follows: Find $(\mathbf{u},\mathsf{p})\in \mathbf{H}_0^1(\Omega)\times  L_0^2(\Omega)$ such that
\begin{equation}\label{eq:problem_variational_forms}
		a(\mathbf{u}; \mathbf{u}, \mathbf{v})+b(\mathbf{v},\mathsf{p}) = \langle \mathbf{f}, \mathbf{v} \rangle
		\quad \forall \mathbf{v} \in \mathbf{H}_0^1(\Omega),
		\qquad
		b(\mathbf{u},\mathsf{q}) = 0
		\quad \forall \mathsf{q} \in L_0^2(\Omega).
\end{equation}
We note that, due to the Rham's theorem \cite[Theorem B.73]{Guermond-Ern}, problem \eqref{eq:problem_variational_forms} is equivalent to the following reduced formulation: Find $\mathbf{u}\in \mathbf{V}(\Omega)$ such that
\begin{align}\label{eq:problem_reduced}
a(\mathbf{u}; \mathbf{u}, \mathbf{v}) =  \langle \mathbf{f}, \mathbf{v} \rangle \quad \forall\mathbf{v}\in \mathbf{V}(\Omega).
\end{align}
An essential ingredient in the proof of \cite[Theorem B.73]{Guermond-Ern} is the fact that the divergence operator is surjective from $\mathbf{H}_0^1(\Omega)$ to $L_0^2(\Omega)$. This implies that there exists a positive constant $\beta$ such that \cite[Chapter I, Section 5.1]{MR851383}, \cite[Corollary B.71]{Guermond-Ern}
\begin{equation}\label{eq:infsup_cont}
\sup_{\mathbf{v}\in \mathbf{H}_0^1(\Omega)}\dfrac{(\mathsf{q},\text{div }\mathbf{v})_{L^2(\Omega)}}{\|\nabla \mathbf{v}\|_{\mathbf{L}^2(\Omega)}}\geq \beta \|\mathsf{q}\|_{L^2(\Omega)}\quad\forall \mathsf{q}\in L_0^2(\Omega).
\end{equation}

Given $\mathbf{f} \in \mathbf{H}^{-1}(\Omega)$, we define
\begin{equation}
\Lambda(\mathbf{f}) :=
1
+
\nu_{+} \nu_{-}^{-1}
+
\mathcal{C}_{N}\nu_{-}^{-2}\|\mathbf{f}\|_{\mathbf{H}^{-1}(\Omega)}
+
\mathcal{C}_{2 \hookrightarrow 2}^2 \nu_{-}^{-1}
+
\mathcal{C}_{s \hookrightarrow 2}^{s} \nu_{-}^{1-s} \|\mathbf{f}\|_{\mathbf{H}^{-1}(\Omega)}^{s-2},
\label{eq:Lambda_f}
\end{equation}
where $\mathcal{C}_{N}$ is as in \eqref{eq:estimates_navier} and $\mathcal{C}_{2 \hookrightarrow 2}$ and $\mathcal{C}_{s \hookrightarrow 2}$ denote the best constants in the Sobolev embeddings $\mathbf{H}^1_0(\Omega) \hookrightarrow \mathbf{L}^2(\Omega)$ and $\mathbf{H}^1_0(\Omega) \hookrightarrow \mathbf{L}^s(\Omega)$, respectively.

We now present an existence result without restrictions on the data.

\begin{theorem}[existence and stability bound]\label{th:BDF_exist}
There exists at least one solution $(\mathbf{u},\mathsf{p})\in \mathbf{H}_0^1(\Omega)\times L_0^2(\Omega)$ for problem \eqref{eq:problem_variational_forms}. Moreover, $(\mathbf{u},\mathsf{p})$ satisfies the bound
\begin{align}
\label{eq:estimate_grad_u}
\|\nabla \mathbf{u}\|_{\mathbf{L}^2(\Omega)}
& \leq \nu_{-}^{-1} \|\mathbf{f}\|_{\mathbf{H}^{-1}(\Omega)},
\\
\label{eq:estimate_p}
\| \mathsf{p} \|_{L^2(\Omega)}
&
\leq
\beta^{-1}\Lambda(\mathbf{f})\|\mathbf{f}\|_{\mathbf{H}^{-1}(\Omega)},
\end{align}
where $\beta$ corresponds to the constant in \eqref{eq:infsup_cont}.
\end{theorem}
\begin{proof}
	Given the equivalence of problems \eqref{eq:problem_variational_forms} and \eqref{eq:problem_reduced}, we analyze problem \eqref{eq:problem_reduced}. We apply the abstract framework developed in \cite[Chapter IV, Theorem 1.2]{MR851383} to our case and divide the proof into three steps.

	\emph{Step 1. Let $\mathbf{w} \in \mathbf{V}(\Omega)$. The bilinear form $a(\mathbf{w};\cdot,\cdot)$ is uniformly $\mathbf{V}(\Omega)$-elliptic.} Let $\mathbf{v}\in\mathbf{V}(\Omega)$. Since $\nu$ is such that $\nu(r) \geq \nu_{-}>0$ for every $r \in \mathbb{R}$, we deduce that
		\begin{align*}
			a(\mathbf{w} ; \mathbf{v}, \mathbf{v})\geq \nu_{-}\|\nabla \mathbf{v}\|_{\mathbf{L}^2(\Omega)}^2+\|\mathbf{v}\|_{\mathbf{L}^2(\Omega)}^2 + ( |\mathbf{w}|^{s-2},|\mathbf{v}|^2 )_{L^2(\Omega)} \geq \nu_{-}\|\nabla \mathbf{v}\|_{\mathbf{L}^2(\Omega)}^2,
		\end{align*}
	where we have used that $a_N(\mathbf{w}; \mathbf{v},\mathbf{v}) = 0$ because $\mathbf{w} \in \mathbf{V}(\Omega)$.

	\emph{Step 2. For all $\mathbf{v} \in \mathbf{V}(\Omega)$, the mapping $\mathbf{u} \mapsto a(\mathbf{u};\mathbf{u},\mathbf{v})$ is sequentially weakly continuous on $\mathbf{V}(\Omega)$.} Let $\{ \mathbf{u}_{k} \}_{k \in \mathbb{N}}$ be a sequence in $\mathbf{V}(\Omega)$ such that $\mathbf{u}_{k}\rightharpoonup\mathbf{u}$ in $\mathbf{V}(\Omega)$ as $k\uparrow \infty$. A compact Sobolev embedding result \cite[Theorem 6.3, Part I]{MR2424078} guarantees that $\mathbf{u}_{k} \rightarrow \mathbf{u}$ in $\mathbf{L}^q(\Omega)$ as $k \uparrow \infty$ for $q<\infty$ in two dimensions and $q<6$ in three dimensions.
	Since $\mathbf{v}\in \mathbf{V}(\Omega)$, it is clear that
	$a_{L}(\mathbf{u}_k,\mathbf{v}) \to a_{L}(\mathbf{u},\mathbf{v})$ as $k\uparrow \infty$. The convergence result $a_{N}(\mathbf{u}_k;\mathbf{u}_k,\mathbf{v})\to a_{N}(\mathbf{u};\mathbf{u},\mathbf{v})$ as $k \uparrow \infty$ can be found in the proof of \cite[Chapter IV, Theorem 2.1]{MR851383}. We now prove that $a_{F}(\mathbf{u}_k;\mathbf{u}_k,\mathbf{v})\to a_{F}(\mathbf{u};\mathbf{u},\mathbf{v})$ as $k\uparrow \infty$. For this purpose, we use \cite[estimate (5.3.33)]{CiarletBook} or \cite[Lemma 5.3]{MR0388811} to obtain
\begin{multline}\label{eq:estimates_a_weakly_cont}
\left|
\int_\Omega (|\mathbf{u}_k|^{s-2}\mathbf{u}_k
- |\mathbf{u}|^{s-2}\mathbf{u})\cdot\mathbf{v} \mathrm{d}x
\right|
\lesssim
\displaystyle
\int_\Omega |\mathbf{u}_k-\mathbf{u}|(|\mathbf{u_k}|
+
|\mathbf{u}|)^{s-2}|\mathbf{v}|
\mathrm{d}x
\\
\displaystyle
\leq \|\mathbf{u_k}
-
\mathbf{u}\|_{\mathbf{L}^{\mu}(\Omega)}\||\mathbf{u_k}|
+
|\mathbf{u}|\|^{s-2}_{\mathbf{L}^{\tau(s-2)}(\Omega)}\|\mathbf{v}\|_{\mathbf{L}^{\mu}(\Omega)}
\displaystyle
\leq \Lambda_{s,\tau}\|\mathbf{u_k}
-
\mathbf{u}\|_{\mathbf{L}^{\mu}(\Omega)}\|\mathbf{v}\|_{\mathbf{L}^{\mu}(\Omega)},
\end{multline}
where $2 \mu^{-1} + \tau^{-1} = 1$ and $\Lambda_{s,\tau} = \Lambda_{s,\tau}(\mathbf{u})$ is such that
\begin{equation}\label{eq:Lambda_st}
	\||\mathbf{u_k}| +  |\mathbf{u}|\|^{s-2}_{\mathbf{L}^{\tau(s-2)}(\Omega)}
	\leq
	\mathcal{C}^{s-2}_{\tau(s-2) \hookrightarrow 2}
	\left[ \mathcal{M} +
	\| \nabla \mathbf{u} \|_{\mathbf{L}^{2}(\Omega)}
	\right]^{s-2}
	=:\Lambda_{s,\tau}(\mathbf{u}).
\end{equation}
Here, $\mathcal{C}_{\tau(s-2) \hookrightarrow 2}$ is the best constant in the Sobolev embedding $\mathbf{H}_0^1(\Omega) \hookrightarrow\mathbf{L}^{\tau(s-2)}(\Omega)$ and $\mathcal{M}$ is such that $ \| \nabla \mathbf{u}_k \|_{\mathbf{L}^{2}(\Omega)} \leq \mathcal{M}$ for every $k \in \mathbb{N}$. The constant $\tau$ is such that $\tau = q$ for some $q > 1$ in two dimensions, and $\tau > 3/2$ is arbitrarily close to $3/2$ in three dimensions. We again invoke the compact embedding of \cite[Theorem 6.3, Part I]{MR2424078} which guarantees that $\mathbf{u}_{k} \rightarrow \mathbf{u}$ in $\mathbf{L}^q(\Omega)$ as $k \uparrow \infty$ for $q<\infty$ in two dimensions and $q<6$ in three dimensions to conclude that $\mathbf{u}_{k}\to\mathbf{u}$ in $\mathbf{L}^\mu(\Omega)$ as $k\uparrow \infty$. As a result, we have obtained that $a_{F}(\mathbf{u}_k;\mathbf{u}_k,\mathbf{v})\to a_{F}(\mathbf{u};\mathbf{u},\mathbf{v})$ as $k\uparrow \infty$ and thus that $a(\mathbf{u}_k;\mathbf{u}_k,\mathbf{v})\to a(\mathbf{u};\mathbf{u},\mathbf{v})$ as $k\uparrow \infty$. The latter shows that $\mathbf{u} \mapsto a(\mathbf{u};\mathbf{u},\mathbf{v})$ is sequentially weakly continuous on $\mathbf{V}(\Omega)$, as we intended to prove.
	
	\emph{Step 3.} The desired result follows from an application of \cite[Chapter IV, Theorem 1.2]{MR851383}. To be precise, we can establish the existence of at least one $\mathbf{u}\in\mathbf{V}(\Omega)$ satisfying the reduced problem \eqref{eq:problem_reduced}. On the other hand, by the inf-sup condition \eqref{eq:infsup_cont}, we obtain that for every solution $\mathbf{u}\in\mathbf{V}(\Omega)$ of \eqref{eq:problem_reduced} there exists a unique $\mathsf{p}\in L_0^2(\Omega)$ such that the pair $(\mathbf{u},\mathsf{p})\in \mathbf{H}_0^1(\Omega)\times L_0^2(\Omega)$ satisfies \eqref{eq:problem_variational_forms}.

	\emph{Step 4.} Let $(\mathbf{u},\mathsf{p})\in \mathbf{H}_0^1(\Omega)\times L_0^2(\Omega)$ be a solution to problem \eqref{eq:problem_variational_forms}. If we set $(\mathbf{v},\mathsf{q}) = (\mathbf{u},0)$ as a test pair in problem \eqref{eq:problem_variational_forms}, we immediately obtain the bound
$
a_L(\mathbf{u},\mathbf{u}) + a_F (\mathbf{u};\mathbf{u},\mathbf{u})
\leq \|\mathbf{f}\|_{\mathbf{H}^{-1}(\Omega)}\|\nabla\mathbf{u}\|_{\mathbf{L}^2(\Omega)},
$
where we have used $\mathbf{u} \in \mathbf{V}(\Omega)$ and $a_N (\mathbf{u};\mathbf{u},\mathbf{u}) = 0$. This bound immediately shows that $\nu_{-}\|\nabla \mathbf{u}\|_{\mathbf{L}^2(\Omega)} \leq \|\mathbf{f}\|_{\mathbf{H}^{-1}(\Omega)}$. The estimate for the pressure, namely estimate \eqref{eq:estimate_p}, follows from the inf-sup condition \eqref{eq:infsup_cont}.
\end{proof}

\subsubsection{Uniqueness} In this section, we provide a uniqueness result for problem \eqref{eq:problem_variational_forms} under a suitable smallness condition.

\begin{theorem}[uniqueness for small data]\label{th:BDF_uniq}
In the framework of Theorem \ref{th:BDF_exist}, if $\mathbf{f}\in\mathbf{H}^{-1}(\Omega)$ is sufficiently small or $\nu$ sufficiently large such that
\begin{align}\label{eq:smalldata}
\nu_{-}^{-2}\mathcal{C}_{N}\|\mathbf{f}\|_{\mathbf{H}^{-1}(\Omega)}< 1,
\end{align}
where $\mathcal{C}_{N}$ is as in \eqref{eq:estimates_navier}, then there is a unique $(\mathbf{u},\mathsf{p})\in \mathbf{H}_0^1(\Omega)\times L_0^2(\Omega)$ that solves \eqref{eq:problem_variational_forms}.
\end{theorem}
\begin{proof}
Let $(\mathbf{u}_{1},\mathsf{p}_1)$ and $(\mathbf{u}_{2},\mathsf{p}_2)$ be two solutions to problem \eqref{eq:problem_variational_forms}. Define $(\mathbf{u},\mathsf{p}):=(\mathbf{u}_{1}-\mathbf{u}_{2},\mathsf{p}_1 - \mathsf{p}_2) \in \mathbf{H}_0^1(\Omega) \times L_0^2(\Omega)$. We note that $(\mathbf{u},\mathsf{p})$ satisfies
\begin{multline}
\label{eq:uniq_error}
a_{L}(\mathbf{u},\mathbf{v})
+
a_{N}(\mathbf{u}_1;\mathbf{u},\mathbf{v})
+
a_{N}(\mathbf{u};\mathbf{u}_{1},\mathbf{v})
-
a_{N}(\mathbf{u};\mathbf{u},\mathbf{v})
+
a_{F}(\mathbf{u}_1;\mathbf{u}_1,\mathbf{v})
\\
-a_{F}(\mathbf{u}_2;\mathbf{u}_2,\mathbf{v})
+
b(\mathbf{v},\mathsf{p})
= 0,
\quad
b(\mathbf{u},\mathsf{q}) = 0
\quad
\forall (\mathbf{v},\mathsf{q}) \in \mathbf{H}_0^1(\Omega) \times L_0^2(\Omega).
\end{multline}
If we set $(\mathbf{v},\mathsf{q}) = (\mathbf{u},0)$ in \eqref{eq:uniq_error} and use the identities in \eqref{eq:properties_navier} and the assumptions on $\nu$, we can conclude that
\begin{equation*}
\label{eq:BDF_uniq_bound_1}
\nu_{-}\|\nabla \mathbf{u}\|_{\mathbf{L}^2(\Omega)}^2+\| \mathbf{u}\|_{\mathbf{L}^2(\Omega)}^2+ a_{F}(\mathbf{u}_1;\mathbf{u}_1,\mathbf{u})-a_{F}(\mathbf{u}_2;\mathbf{u}_2,\mathbf{u})\leq -a_{N}(\mathbf{u};\mathbf{u}_1,\mathbf{u}).
\end{equation*}
Let us use \cite[Lemma 4.4, Chapter I]{MR1230384} to obtain $a_{F}(\mathbf{u}_1;\mathbf{u}_1,\mathbf{u})-a_{F}(\mathbf{u}_2;\mathbf{u}_2,\mathbf{u}) \geq 0$. This and the bounds \eqref{eq:estimates_navier} and \eqref{eq:estimate_grad_u} allows us to deduce that
\begin{equation*}
\label{eq:BDF_uniq_2}
\nu_{-}\|\nabla \mathbf{u}\|_{\mathbf{L}^2(\Omega)}^2 \leq \mathcal{C}_{N}\|\nabla \mathbf{u}\|_{\mathbf{L}^2(\Omega)}^2 \|\nabla \mathbf{u}_1\|_{\mathbf{L}^2(\Omega)}\leq \nu_{-}^{-1} \mathcal{C}_{N} \|\nabla \mathbf{u}\|_{\mathbf{L}^2(\Omega)}^2 \|\mathbf{f}\|_{\mathbf{H}^{-1}(\Omega)}.
\end{equation*}
From this and from \eqref{eq:smalldata}, we can deduce that $\mathbf{u} = \mathbf{0}$. Finally, we use problem \eqref{eq:uniq_error} again to establish that $b(\mathbf{v},\mathsf{p}) = 0$ for any $\mathbf{v} \in \mathbf{H}_0^1(\Omega)$. The inf-sup condition \eqref{eq:infsup_cont} allows us to conclude that $\mathsf{p} = 0$.
\end{proof}

\subsection{A nonlinear heat equation}
\label{sec:heat_equation_problem}
We now examine existence and uniqueness results for a \emph{nonlinear} heat equation with convection. For this purpose, we consider $g \in H^{-1}(\Omega)$ and a thermal diffusion coefficient $\kappa \in C^{0,1}(\mathbb{R})$ that satisfies $0<\kappa_{-}\leq \kappa(r)\leq \kappa_{+}$ for each $r \in \mathbb{R}$, where $\kappa_{+} \geq \kappa_{-}>0$. We denote by $\mathcal{L}_{\kappa}$ the Lipschitz constant of $\kappa$. A weak formulation for the \emph{nonlinear} heat equation is as follows:
\begin{equation}\label{eq:heat_problem}
T\in H_0^1(\Omega):
\quad
	\int_\Omega\left(\kappa(T)\nabla T\cdot \nabla S + (\mathbf{v}\cdot\nabla T)S\right)\mathrm{d}x= \langle g,S\rangle \quad \forall S\in H_0^{1}(\Omega),
\end{equation}
where $\mathbf{v} \in \mathbf{V}(\Omega)$. As an instrumental ingredient to perform an analysis, we introduce the map $\mathcal{A}: H^1_0(\Omega) \rightarrow H^{-1}(\Omega)$, which is defined by
\begin{align}\label{eq:mapA_temp}
\langle \mathcal{A}(T),S\rangle:=\int_{\Omega}\left(\kappa(T)\nabla T\cdot\nabla S +(\mathbf{v}\cdot\nabla T) S\right)\mathrm{d}x\quad \forall T,S \in H_0^1(\Omega).
\end{align}

The existence of solutions to problem \eqref{eq:heat_problem} is as follows.

\begin{theorem}[existence]\label{th:heat_existence}
There exists at least one solution $T \in H^1_0(\Omega)$ to \eqref{eq:heat_problem}.
\end{theorem}

\begin{proof}
Inspired by \cite[Section 2.4]{MR3014456}, we apply the theory of pseudo-monotone operators \cite[Theorem 2.6]{MR3014456}: \emph{Any pseudo-monotone and coercive map is surjective.} Let us thus prove that the bounded map $A$ defined in \eqref{eq:mapA_temp} is pseudo-monotone and coercive. To accomplish this task, we define the functions $\textswab{a}, \textswab{c}: \Omega \times \mathbb{R} \times \mathbb{R}^d \rightarrow \mathbb{R}$ as
$
\textswab{a}(x,r,\mathbf{s}):=\kappa(r)\mathbf{s}
$
and
$
\textswab{c}(x,r,\mathbf{s}) := \mathbf{v} \cdot \mathbf{s}.
$
Since $\textswab{a}(x,T,\nabla T)=\kappa(T)\nabla T$ and $\textswab{c}(x,T,\nabla T)= \mathbf{v} \cdot \nabla T$, we can rewrite \eqref{eq:mapA_temp} as follows:
\begin{align*}
\langle \mathcal{A}(T),S\rangle=\int_{\Omega}\left(\textswab{a}(x,T,\nabla T)\cdot \nabla S+\textswab{c}(x,T,\nabla T)S\right)\mathrm{d}x.
\end{align*}

We now proceed on the basis of three steps.
	
	\emph{Step 1.} \emph{$\mathcal{A}$ is coercive.} Let $S \in H^1_0(\Omega)$. Since $\mathbf{v}\in\mathbf{V}(\Omega)$, we immediately arrive at the coercivity of $\mathcal{A}$  because $\langle\mathcal{A}(S),S\rangle \geq \kappa_- \|\nabla S\|^2_{\mathbf{L}^2(\Omega)}$.

	\emph{Step 2.} \emph{$\mathcal{A}$ is pseudo-monotone}. Since $\mathcal{A}$ is bounded, to prove that $\mathcal{A}$ is pseudo-monotone, we prove that $\mathcal{A}$ satisfies the conditions in \cite[Lemma 2.32]{MR3014456}. First, we note that the functions $\textswab{a}$ and $\textswab{c}$ are Carath\'eodory functions. Second, $\mathcal{A}$ is \emph{monotone in the main part}, i.e., the function $\textswab{a}$ satisfies the following condition:
\begin{align*}
	(\textswab{a}(x,r,\mathbf{s})-\textswab{a}(x,r,\tilde{\mathbf{s}}))\cdot (\mathbf{s}-\tilde{\mathbf{s}})= \kappa(r)|\mathbf{s}-\tilde{\mathbf{s}}|^2\geq 0
	\quad
	\mathrm{a.e.}~x\in \Omega,
	\,
	\forall r\in\mathbb{R},
	\,
	\forall \mathbf{s}, \tilde{\mathbf{s}} \in \mathbb{R}^d.
	\end{align*}
	Thirdly, $\textswab{c}$ is linearly dependent on $\mathbf{s}$. Finally, since $\textswab{a}(x,r,\mathbf{s})=\kappa(r)\mathbf{s}$ and $\kappa$ is uniformly bounded, it follows immediately that $|\textswab{a}(x,r,\mathbf{s})| \leq \kappa_{+} |\mathbf{s}|$ and that $\textswab{a}$ thus satisfies  \cite[(2.55a)]{MR3014456}. On the other hand, as a consequence of suitable Young's inequalities, it can be proved that $\textswab{c}$ satisfies \cite[(2.55c)]{MR3014456}. An application of \cite[Lemma 2.32]{MR3014456} thus shows that $\mathcal{A}$ is pseudo-monotone.

	\emph{Step 3.} We apply \cite[Theorem 2.6]{MR3014456} to conclude that $\mathcal{A}$ is surjective, in other words, for $g\in H^{-1}(\Omega)$ there is a solution to problem \eqref{eq:heat_problem}. This concludes the proof.
\end{proof} 

Global uniqueness can be obtained when the solution is slightly smoother and the datum is suitably restricted.

\begin{theorem}[uniqueness for small data]\label{th:heat_uniqueness}
Let us assume that \eqref{eq:heat_problem} has a solution $T_1 \in W^{1,3}(\Omega)\cap H_0^1(\Omega)$ such that
\begin{equation}\label{eq:heat_estimates_uniq}
\kappa_{-}^{-1}\mathcal{L}_{\kappa}\mathcal{C}_{6 \hookrightarrow 2}\|\nabla T_1\|_{\mathbf{L}^{3}(\Omega)} <1.
\end{equation}
Then problem \eqref{eq:heat_problem} has no other solution $T_2$ in $H_{0}^1(\Omega)$. Here, $\mathcal{C}_{6 \hookrightarrow 2}$ is the best constant in the Sobolev embedding $H^1_0(\Omega) \hookrightarrow L^6(\Omega)$ and $\mathcal{L}_\kappa$ is the Lipschitz constant of $\kappa$.
\end{theorem}

\begin{proof}
We proceed similarly to the proof of Theorem \ref{th:BDF_uniq}. Let $T_2$ be another solution of \eqref{eq:heat_problem}. A simple calculation shows that $T:=T_1-T_2\in H_0^1(\Omega)$ verifies the identity
	\begin{equation}\label{eq:heat_error}
		\int_\Omega\left( \kappa(T_2)\nabla T\cdot \nabla S +  (\mathbf{v}\cdot\nabla T)S\right)\mathrm{d}x =\int_\Omega (\kappa(T_2)- \kappa(T_1))\nabla T_1\cdot\nabla S\mathrm{d}x
	\end{equation}
for every $S\in H_0^1(\Omega)$. We now estimate the term $\|\nabla T\|_{{\mathbf{L}}^2(\Omega)}$. To this end, we set $S=T$ in \eqref{eq:heat_error}, use the fact that $\mathbf{v}\in \mathbf{V}(\Omega)$, apply the Lipschitz condition of $\kappa$, and apply the standard Sobolev embedding $H^1_0(\Omega) \hookrightarrow L^6(\Omega)$ to conclude that
$
\kappa_{-}\|\nabla T\|_{{\mathbf{L}}^2(\Omega)}\leq \mathcal{L}_{\kappa}\mathcal{C}_{6 \hookrightarrow 2}\|\nabla T\|_{{\mathbf{L}}^2(\Omega)}\| \nabla T_1\|_{{\mathbf{L}}^3(\Omega)}.
$
By substituting the bound \eqref{eq:heat_estimates_uniq}, we then conclude $\|\nabla T\|_{{\mathbf{L}}^2(\Omega)}\leq 0$. This proves uniqueness when \eqref{eq:heat_estimates_uniq} holds.
\end{proof}

\begin{remark}[$d=2$]
If $d=2$, the assumption on $T_1$ in \eqref{eq:heat_estimates_uniq} can be improved to
$\mathcal{L}_{\kappa}\mathcal{C}_{\sigma \hookrightarrow 2}\|\nabla T_1\|_{\mathbf{L}^{t}(\Omega)}/\kappa_{-} <1$ for some $t>2$, where $\sigma$ satisfies $1/\sigma + 1 /t = 1/2$. This is achieved by exploiting the fact that $H^1_0(\Omega) \hookrightarrow L^{\iota}(\Omega)$ for $\iota < \infty$ in two dimensions.
\end{remark}


\section{The coupled problem}
\label{sec:coupled_problem}
The main goal of this section is to show the existence of suitable weak solutions for the coupled problem \eqref{eq:model} and to derive a uniqueness result under suitable assumptions. In a first step, we introduce the assumptions under which we will work and introduce the concept of a weak solution.

\subsection{Main assumptions}\label{sec:main_assump}

Inspired by the results in the previous sections, we consider the following assumptions on the viscosity and diffusion coefficients.

$\bullet$ \emph{Viscosity:} The viscosity $\nu\in C^{0,1}(\mathbb{R})$ is a function that is strictly positive and bounded: there exist positive constants $\nu_{-}$ and $\nu_{+}$ such that
\begin{equation}\label{eq:nu}
\nu_{-}\leq \nu(r)\leq \nu_{+} \quad \forall r\in\mathbb{R}.
\end{equation}

$\bullet$ \emph{Diffusivity:} The thermal coefficient $\kappa\in C^{0,1}(\mathbb{R})$ is a strictly positive and bounded function: there exists positive constant $\kappa_{-}$ and $\kappa_{+}$ such that
\begin{equation}\label{eq:kappa}
\kappa_{-}\leq \kappa(r)\leq \kappa_{+} \quad \forall r\in\mathbb{R}.
\end{equation}

$\bullet$ \emph{Lipschitz constants:} We denote by $\mathcal{L}_{\nu}$ and $\mathcal{L}_{\kappa}$ the Lipschitz constants associated to the functions $\nu$ and $\kappa$, respectively.

\subsection{Weak solution}
We use the following notion of weak solution for \eqref{eq:model}.

\begin{definition}[weak solution]
Let $\mathbf{f} \in \mathbf{H}^{-1}(\Omega)$ and let $g \in H^{-1}(\Omega)$. We say that $(\mathbf{u},\mathsf{p},T)\in \mathbf{H}_0^1(\Omega)\times L_0^2(\Omega)\times H_0^1(\Omega)$ is a weak solution to \eqref{eq:model} if
\begin{equation}\label{eq:modelweak}
\begin{array}{c}
\displaystyle
\int_\Omega\left(\nu(T)\nabla\mathbf{u}\cdot\nabla\mathbf{v}+(\mathbf{u}\cdot\nabla)\mathbf{u}\cdot\mathbf{v}+ \mathbf{u}\cdot\mathbf{v} +|\mathbf{u}|^{s-2}\mathbf{u}\cdot\mathbf{v}-\mathsf{p}\,\mathrm{div}\,\mathbf{v} \right)\mathrm{d}x
=
\displaystyle \langle\mathbf{f},\mathbf{v}\rangle,
\\
\displaystyle
\int_{\Omega}\mathsf{q}\,\mathrm{div}\,\mathbf{u}\,\mathrm{d}x=0,
\qquad
\displaystyle\int_\Omega\left(\kappa(T)\nabla T\cdot \nabla S+(\mathbf{u}\cdot\nabla T)S\right)\mathrm{d}x
=
\displaystyle\langle g,S\rangle,
\end{array}
\end{equation}
for all $(\mathbf{v},\mathsf{q},S)\in \mathbf{H}_0^1(\Omega)\times L_0^2(\Omega)\times H_0^1(\Omega)$. The parameter $s$ belongs to $[3,4]$.
\end{definition}
It is important to note that, given the assumptions imposed on the problem data, all terms in system \eqref{eq:modelweak} are well-defined.

\subsection{Existence of solutions}
We are now in a position to establish an existence result for problem \eqref{eq:modelweak}. To simplify our presentation, we use the following notation:
\begin{align*}
& \forall
(\mathbf{v},S) \in \mathbf{H}_0^1(\Omega) \times H_0^1(\Omega):
\quad
\| (\mathbf{v},S) \|_{ \mathbf{H}_0^1(\Omega) \times H_0^1(\Omega) }
:=
\left[
\|\nabla \mathbf{v}\|_{\mathbf{L}^2(\Omega)}^2+\|\nabla S\|_{\mathbf{L}^2(\Omega)}^2
\right]^{\frac{1}{2}},
\\
& \forall
(\mathbf{f},g) \in \mathbf{H}^{-1}(\Omega) \times H^{-1}(\Omega):
\quad
\| (\mathbf{f},g) \|_{\mathbf{H}^{-1}(\Omega) \times H^{-1}(\Omega)}
:=\left[
\|\mathbf{f}\|_{\mathbf{H}^{-1}(\Omega)}^2+\!\|g\|_{H^{-1}(\Omega)}^2\right]^{\frac{1}{2}}.
\end{align*}

\begin{theorem}[existence of solutions]\label{theorem_coupled_existence}
Let $d \in \{2,3\}$, and let $\Omega \subset \mathbb{R}^d$ be an open and bounded domain with Lipschitz boundary $\partial\Omega$. Let $\nu$ and $\kappa$ in $C^{0,1}(\mathbb{R})$ be such that inequalities \eqref{eq:nu} and \eqref{eq:kappa} hold. If $\mathbf{f} \in \mathbf{H}^{-1}(\Omega)$ and $g \in H^{-1}(\Omega)$, then problem \eqref{eq:modelweak} has at least one solution $(\mathbf{u},\mathsf{p},T) \in \mathbf{H}^1_0(\Omega) \times L_0^2(\Omega) \times H^1_0(\Omega)$. Moreover, we have
\begin{align}
	\label{eq:estimate_grad_u_coupled}
	\|\nabla \mathbf{u}\|_{\mathbf{L}^2(\Omega)}
	& \leq \nu_{-}^{-1} \|\mathbf{f}\|_{\mathbf{H}^{-1}(\Omega)},
	\quad
	\| \mathsf{p} \|_{L^2(\Omega)}
	\leq {\beta}^{-1} \Lambda(\mathbf{f})\|\mathbf{f}\|_{\mathbf{H}^{-1}(\Omega)},
	\\
	\label{eq:estimate_T_coupled}
	\|\nabla T \|_{\mathbf{L}^2(\Omega)}
	& \leq \kappa_{-}^{-1} \| g \|_{H^{-1}(\Omega)},
\end{align}
where $\Lambda$ is defined in \eqref{eq:Lambda_f}.
\label{th:existence_coupled_problem}
\end{theorem}
\begin{proof}
We adapt the proof of Theorem 2.2 in \cite{MR3342221} to our case and divide the proof into several steps.
	
\emph{Step 1.} \emph{A mapping $\Phi$:} Let $\mathbf{u}, \mathbf{v} \in \mathbf{V}(\Omega)$, and let $T,S \in H_0^1(\Omega)$ be arbitrary. To simplify the notation, we define the variables $\boldsymbol{\mathcal{U}}:=(\mathbf{u},T)$ and $\boldsymbol{\mathcal{V}}:=(\mathbf{v},S)$. We introduce the mapping $\Phi$ from the space $\mathbf{V}(\Omega) \times H^1_0(\Omega)$ into its dual space as
\begin{multline}
\label{eq:definition_Phi}
\langle \Phi(\boldsymbol{\mathcal{U}}),\boldsymbol{\mathcal{V}} \rangle
:=
\int_{\Omega}\left( \nu(T)\nabla \mathbf{u} \cdot \nabla \mathbf{v}
+
(\mathbf{u}\cdot\nabla)\mathbf{u}\cdot\mathbf{v}
+
\mathbf{u}\cdot\mathbf{v}
+
|\mathbf{u}|^{s-2}\mathbf{u}\cdot\mathbf{v} \right)\mathrm{d}x
\\
+
\int_\Omega\left(\kappa(T)\nabla T\cdot \nabla S + (\mathbf{u}\cdot\nabla T)S\right)\mathrm{d}x  - \langle\mathbf{f},\mathbf{v}\rangle - \langle g, S\rangle.
\end{multline}
As a consequence of $\mathbf{H}_0^1(\Omega) \hookrightarrow \mathbf{L}^{\iota}(\Omega)$, which holds for $\iota < \infty$ when $d=2$ and $\iota \leq 6$ when $d=3$, suitable H\"older inequalities, the bounds \eqref{eq:estimate_for_Forchheimer_2} and \eqref{eq:estimates_navier}, and the properties that $\nu$ and $\kappa$ satisfy, we can deduce that $\Phi$ is continuous on $\mathbf{V}(\Omega) \times H^1_0(\Omega)$.

If, on the other hand, we replace $\boldsymbol{\mathcal{U}}$ by $\boldsymbol{\mathcal{V}}$ in \eqref{eq:definition_Phi} and use that $a_N(\mathbf{v};\mathbf{v},\mathbf{v}) = 0$, $a_F(\mathbf{v};\mathbf{v},\mathbf{v}) = ( |\mathbf{v}|^{s-2} , |\mathbf{v}|^{2})_{L^2(\Omega)} \geq 0$, and $\int_{\Omega} (\mathbf{v} \cdot \nabla S)S\mathrm{d}x = 0$, together with the properties \eqref{eq:nu} and \eqref{eq:kappa}, we conclude that
\begin{multline*}
\langle
\Phi(\boldsymbol{\mathcal{V}}),\boldsymbol{\mathcal{V}}
\rangle
\geq
\min \left\{\nu_-, \kappa_-\right\}
\| (\mathbf{v},S) \|^2_{\mathbf{H}_0^1(\Omega) \times H_0^1(\Omega)}
\\
-
\| (\mathbf{f},g) \|_{ \mathbf{H}^{-1}(\Omega) \times H^{-1}(\Omega) }
\| (\mathbf{v},S) \|_{ \mathbf{H}_0^1(\Omega) \times H_0^1(\Omega) }.
	\end{multline*}
From this, we can conclude that $\langle \Phi(\boldsymbol{\mathcal{V}}),\boldsymbol{\mathcal{V}} \rangle \geq 0$ for all $\boldsymbol{\mathcal{V}} \in \mathbf{V}(\Omega) \times H^1_0(\Omega)$ such that
$
\| (\mathbf{v},S) \|_{ \mathbf{H}_0^1(\Omega) \times H_0^1(\Omega) }
=
\min
\{\nu_-, \kappa_- \}^{-1}
\| (\mathbf{f},g) \|_{ \mathbf{H}^{-1}(\Omega) \times H^{-1}(\Omega) }
=: \delta.
$

\emph{Step 2.} \emph{Galerkin approximation} \cite[Theorem 2.6]{MR3014456}: Since $\mathbf{V}(\Omega)$ is separable, we can take a sequence $\{ \mathbb{V}_k \}_{k \in \mathbb{N}}$ of finite-dimensional subspaces of $\mathbf{V}(\Omega)$ such that
\begin{equation}
 \forall k \in \mathbb{N}:
 \quad
 \mathbb{V}_{k} \subset \mathbb{V}_{k+1} \subset \mathbf{V}(\Omega),
 \quad
 \cup \{ \mathbb{V}_k: k \in \mathbb{N} \} \textrm{ is dense in } \mathbf{V}(\Omega).
 \label{eq:density}
\end{equation}
We can also take a sequence $\{ \mathbb{W}_k \}_{k \in \mathbb{N}}$ of finite-dimensional subspaces of $H_0^1(\Omega)$ so that the properties in \eqref{eq:density} hold by replacing $\mathbb{V}_k$ and $\mathbf{V}(\Omega)$ by $\mathbb{W}_k$ and $H_0^1(\Omega)$, respectively. Then, we define a Galerkin approximation $\boldsymbol{\mathcal{U}}_k = (\mathbf{u}_k,T_k) \in \mathbb{V}_k \times \mathbb{W}_k$ by the identity
\[
 \langle \Phi(\boldsymbol{\mathcal{U}}_k),\boldsymbol{\mathcal{V}}_k \rangle = 0
 \quad
 \forall \boldsymbol{\mathcal{V}}_k \in \mathbb{V}_k \times \mathbb{W}_k.
\]
As shown in step 1 on a continuous level, $\Phi$ is a continuous from $\mathbb{V}_k \times \mathbb{W}_k$ into $\mathbb{V}_k \times \mathbb{W}_k$ and satisfies the property $\langle \Phi(\boldsymbol{\mathcal{V}}_k),\boldsymbol{\mathcal{V}}_k \rangle \geq 0$ for all $\boldsymbol{\mathcal{V}}_k \in \mathbb{V}_k \times \mathbb{W}_k$ such that
$
\| (\mathbf{v}_k,S_k) \|_{ \mathbf{H}_0^1(\Omega) \times H_0^1(\Omega) }
=
\min \{\nu_-, \kappa_- \}^{-1}
\| (\mathbf{f},g) \|_{ \mathbf{H}^{-1}(\Omega) \times H^{-1}(\Omega) }
= \delta.
$
Applying a consequence of Brouwer's classical fixed-point theorem \cite[Chapter IV, Corollary 1.1]{MR851383}, we deduce the existence of a solution $\boldsymbol{\mathcal{U}}_k$ such that
\begin{equation}
 \label{eq:Galerkin_approximation}
 \langle \Phi(\boldsymbol{\mathcal{U}}_k),\boldsymbol{\mathcal{V}}_k \rangle = 0
 \quad
 \forall \boldsymbol{\mathcal{V}}_k \in \mathbb{V}_k \times \mathbb{W}_k,
 \qquad
 \| (\mathbf{u}_k,T_k) \|_{ \mathbf{H}_0^1(\Omega) \times H_0^1(\Omega) } \leq \delta.
\end{equation}

\emph{Step 3.} \emph{Limite passage:} Since the sequences $\{ \mathbf{u}_k \}_{k \in \mathbb{N}}$ and $\{ T_k \}_{k \in \mathbb{N}}$ are uniformly bounded in $\mathbf{H}_0^1(\Omega)$ and $H_0^1(\Omega)$, respectively, we deduce the existence of nonrelabeled subsequences such that $\mathbf{u}_k \rightharpoonup \mathbf{u}$ and $T_k \rightharpoonup T$ in $\mathbf{H}_0^1(\Omega)$ and $H_0^1(\Omega)$, respectively, as $k \uparrow \infty$. The compact embedding of \cite[Theorem 6.3, Part I]{MR2424078} guarantees that $\mathbf{u}_{k} \rightarrow \mathbf{u}$ in $\mathbf{L}^q(\Omega)$ and $T_{k} \rightarrow T$ in $L^q(\Omega)$ as $k \uparrow \infty$ for $q<\infty$ in two dimensions and $q<6$ in three dimensions. Note that by construction, for every $\ell \leq k$ we have that
\begin{equation}
\label{eq:Galerkin_approximation_ell}
\boldsymbol{\mathcal{U}}_k = (\mathbf{u}_k,T_k) \in \mathbb{V}_{k} \times \mathbb{W}_{k}:
\quad
\langle \Phi(\boldsymbol{\mathcal{U}}_k),\boldsymbol{\mathcal{V}}_{\ell} \rangle = 0
\quad
\forall \boldsymbol{\mathcal{V}}_{\ell} \in \mathbb{V}_{\ell} \times \mathbb{W}_{\ell}.
\end{equation}

Let us now prove that the limit point $(\mathbf{u},T)$ solves \eqref{eq:modelweak}. To do so, we note that:

(i) \emph{$\int_{\Omega} \nu(T_k)\nabla \mathbf{u}_k \cdot \nabla \mathbf{v} \mathrm{d}x \rightarrow \int_{\Omega} \nu(T) \nabla \mathbf{u}\cdot \nabla \mathbf{v} \mathrm{d}x$ as $k \uparrow \infty$ for $\mathbf{v} \in \mathbf{V}(\Omega)$:} It follows from the strong convergence of $\{ T_k \}_{k \in \mathbb{N}}$ to $T$ in $L^{2}(\Omega)$ and the Lipschitz continuity of $\nu$ that $\nu(T_k) \nabla \mathbf{v} \rightarrow \nu(T) \nabla \mathbf{v}$ almost everywhere in $\Omega$ as $k \uparrow \infty$. Since $\nu$ is bounded, the Lebesgue dominated convergence theorem shows that $\nu(T_k) \nabla \mathbf{v} \rightarrow \nu(T)\nabla \mathbf{v}$ in $\mathbf{L}^2(\Omega)$ as $k \uparrow \infty$. As a result, $\int_{\Omega} \nu(T_k)\nabla \mathbf{u}_k \cdot \nabla \mathbf{v} \mathrm{d}x \rightarrow \int_{\Omega} \nu(T) \nabla \mathbf{u} \cdot \nabla \mathbf{v} \mathrm{d}x$ as $k \uparrow \infty$.

(ii) \emph{$a_N(\mathbf{u}_k;\mathbf{u}_k,\mathbf{v}) \rightarrow a_N(\mathbf{u};\mathbf{u},\mathbf{v})$ as $k \uparrow \infty$ for $\mathbf{v} \in \mathbf{V}(\Omega)$}: The proof of this convergence result can be found in the proof of \cite[Chapter IV, Theorem 2.1]{MR851383}.

(iii) \emph{$a_F(\mathbf{u}_k;\mathbf{u}_k,\mathbf{v}) \rightarrow a_F(\mathbf{u};\mathbf{u},\mathbf{v})$ as $k \uparrow \infty$ for $\mathbf{v} \in \mathbf{V}(\Omega)$}: We begin with an application of \cite[estimate (5.3.33)]{CiarletBook} or \cite[Lemma 5.3]{MR0388811} and obtain
\begin{multline}\label{eq:estimates_a_weakly_cont_new}
\left|
\int_\Omega (|\mathbf{u}_k|^{s-2}\mathbf{u}_k
- |\mathbf{u}|^{s-2}\mathbf{u})\cdot\mathbf{v} \mathrm{d}x
\right|
\lesssim
\displaystyle
\int_\Omega |\mathbf{u}_k-\mathbf{u}|(|\mathbf{u_k}|
+
|\mathbf{u}|)^{s-2}|\mathbf{v}|
\mathrm{d}x
\\
\displaystyle
\leq \|\mathbf{u_k}
-
\mathbf{u}\|_{\mathbf{L}^{\mu}(\Omega)}\||\mathbf{u_k}|
+
|\mathbf{u}|\|^{s-2}_{\mathbf{L}^{\tau(s-2)}(\Omega)}\|\mathbf{v}\|_{\mathbf{L}^{\mu}(\Omega)}
\displaystyle
\leq \Lambda_{s,\tau}\|\mathbf{u_k}
-
\mathbf{u}\|_{\mathbf{L}^{\mu}(\Omega)}\|\mathbf{v}\|_{\mathbf{L}^{\mu}(\Omega)},
\end{multline}
where $2 \mu^{-1} + \tau^{-1} = 1$ and $\Lambda_{s,\tau} = \Lambda_{s,\tau}(\mathbf{u})$ is such that
\begin{equation}\label{eq:Lambda_st_new}
	\||\mathbf{u_k}| +  |\mathbf{u}|\|^{s-2}_{\mathbf{L}^{\tau(s-2)}(\Omega)}
	\leq
	\mathcal{C}^{s-2}_{\tau(s-2) \hookrightarrow 2}
	\left[ \mathcal{M} +
	\| \nabla \mathbf{u} \|_{\mathbf{L}^{2}(\Omega)}
	\right]^{s-2}
	=:\Lambda_{s,\tau}(\mathbf{u}).
\end{equation}
Here, $\mathcal{C}_{\tau(s-2) \hookrightarrow 2}$ is the best constant in the Sobolev embedding $\mathbf{H}_0^1(\Omega) \hookrightarrow\mathbf{L}^{\tau(s-2)}(\Omega)$ and $\mathcal{M}$ is such that $ \| \nabla \mathbf{u}_k \|_{\mathbf{L}^{2}(\Omega)} \leq \mathcal{M}$ for every $k \in \mathbb{N}$. The constant $\tau$ is such that $\tau = q$ for some $q > 1$ in two dimensions, and $\tau > 3/2$ is arbitrarily close to $3/2$ in three dimensions. We again invoke the compact embedding of \cite[Theorem 6.3, Part I]{MR2424078} which guarantees that $\mathbf{u}_{k} \rightarrow \mathbf{u}$ in $\mathbf{L}^q(\Omega)$ as $k \uparrow \infty$ for $q<\infty$ in two dimensions and $q<6$ in three dimensions to conclude that $\mathbf{u}_{k}\to\mathbf{u}$ in $\mathbf{L}^\mu(\Omega)$ as $k\uparrow \infty$. As a result, we have obtained that $a_{F}(\mathbf{u}_k;\mathbf{u}_k,\mathbf{v})\to a_{F}(\mathbf{u};\mathbf{u},\mathbf{v})$ as $k\uparrow \infty$.

(iv) \emph{$\int_{\Omega} \kappa(T_k) \nabla T_k \cdot \nabla S \mathrm{d}x \rightarrow \int_{\Omega} \kappa(T) \nabla T \cdot \nabla S \mathrm{d}x$ as $k \uparrow \infty$ for $S \in H_0^1(\Omega)$:} This convergence result follows from arguments similar to those used in item (i).

The results obtained in (i)--(iv) allow us to conclude that
$
\langle \Phi(\boldsymbol{\mathcal{U}}),\boldsymbol{\mathcal{V}}_{\ell} \rangle = 0
$
for all $\boldsymbol{\mathcal{V}}_{\ell} \in \mathbb{V}_{\ell} \times \mathbb{W}_{\ell}$
and by density that
$
\langle \Phi(\boldsymbol{\mathcal{U}}),\boldsymbol{\mathcal{V}} \rangle = 0
$
for all $\boldsymbol{\mathcal{V}} \in \mathbf{V}(\Omega) \times H_0^1(\Omega)$. We can therefore conclude that $(\mathbf{u},T)$ satisfies the second and third equations in \eqref{eq:modelweak} and
\begin{align*}
 \displaystyle
 \langle \mathcal{H}, \mathbf{v} \rangle :=
\int_\Omega\left(\nu(T)\nabla\mathbf{u}\cdot\nabla\mathbf{v}+(\mathbf{u}\cdot\nabla)\mathbf{u}\cdot\mathbf{v}+ \mathbf{u}\cdot\mathbf{v} +|\mathbf{u}|^{s-2}\mathbf{u}\cdot\mathbf{v}
\right)\mathrm{d}x
-
\displaystyle \langle\mathbf{f},\mathbf{v}\rangle
= 0,	
\end{align*}
for all $\mathbf{v} \in \mathbf{V}(\Omega)$.

\emph{Step 4.} \emph{The pressure:} The functional $\mathcal{H}$ is linear and continuous on $\mathbf{H}_0^1(\Omega)$ and is zero on the space $\mathbf{V}(\Omega)$. Consequently, by virtue of de Rham's theorem \cite[Theorem B.73]{Guermond-Ern}, there exists $\mathsf{p} \in L_0^2(\Omega)$ such that $\langle \mathcal{H}, \mathbf{v} \rangle = \langle \nabla \mathsf{p}, \mathbf{v}\rangle$ for $\mathbf{v} \in \mathbf{H}_0^1(\Omega)$. We have thus proved the existence of a solution $(\mathbf{u},\mathsf{p},T)$.

\emph{Step 5.} \emph{Stability bounds:} Let $(\mathbf{u},\mathsf{p},T)\in \mathbf{H}_0^1(\Omega)\times L_0^2(\Omega) \times H_0^1(\Omega)$ be a solution of problem \eqref{eq:modelweak}. If we use $\mathbf{v} = \mathbf{u}$ as the test function in the first equation of problem \eqref{eq:modelweak}, we immediately obtain the bound
\[
\nu_{-}\| \nabla \mathbf{u}\|_{\mathbf{L}^2(\Omega)}^2+\| \mathbf{u}\|_{\mathbf{L}^2(\Omega)}^2+ \| \mathbf{u}\|^s_{\mathbf{L}^s(\Omega)}
\leq \|\mathbf{f}\|_{\mathbf{H}^{-1}(\Omega)}\|\nabla\mathbf{u}\|_{\mathbf{L}^2(\Omega)},
\]
where we have used that $a_N (\mathbf{u};\mathbf{u},\mathbf{u}) = 0$ because $\mathbf{u} \in \mathbf{V}(\Omega)$. This bound shows that $\nu_{-}\|\nabla \mathbf{u}\|_{\mathbf{L}^2(\Omega)} \leq \|\mathbf{f}\|_{\mathbf{H}^{-1}(\Omega)}$. The stability bound for the temperature variable in \eqref{eq:estimate_T_coupled} follows similar arguments. The estimate for the pressure follows from the inf-sup condition \eqref{eq:infsup_cont}. This concludes the proof.
\end{proof}
\subsection{Uniqueness of solutions}

\emph{Without regularity assumptions on the solution, the derivation of uniqueness of solutions for the nonlinear system \eqref{eq:modelweak} appears problematic.} Let us show the uniqueness of solutions under suitable assumptions.

\begin{theorem}[uniqueness of solutions]\label{thm:uniq_coupled_3D}
Let the assumptions of Theorem \ref{th:existence_coupled_problem} hold. Let $d=3$ and assume that problem \eqref{eq:modelweak} has a solution $(\mathbf{u}_1,\mathsf{p}_1,T_1)\in\mathbf{W}^{1,3}(\Omega) \cap \mathbf{H}^1_0 (\Omega) \times L^2_0(\Omega) \times W^{1,3}(\Omega) \cap H^1_0 (\Omega)$ such that
\begin{align}
\label{eq:small_data_coupled}
\frac{\mathcal{C}_N}{\nu_{-}^2}\|\mathbf{f}\|_{\mathbf{H}^{-1}(\Omega)} + \frac{\mathcal{L}_{\nu}\mathcal{C}_{6 \hookrightarrow 2}\mathcal{C}^2_{4\hookrightarrow 2}\|g\|_{H^{-1}(\Omega)}\| \nabla \mathbf{u}_1 \|_{\mathbf{L}^3(\Omega)}}{\nu_{-}\kappa_{-}\left( \kappa_{-} - \mathcal{L}_{\kappa}\mathcal{C}_{6 \hookrightarrow 2} \| \nabla T_1 \|_{{\mathbf{L}}^3(\Omega)} \right)}  & < 1,
\\
\label{eq:small_data_coupled_2}
\mathcal{L}_{\kappa}\mathcal{C}_{6 \hookrightarrow 2}\| \nabla T_1 \|_{{\mathbf{L}}^3(\Omega)}  & <  \kappa_{-}.
\end{align}
Then, \eqref{eq:modelweak} has no other solution $(\mathbf{u_2},\mathsf{p}_2,T_2)\in\mathbf{H}^1_0(\Omega) \times L^2_0(\Omega) \times H^1_0(\Omega)$. Here, $\mathcal{C}_{6 \hookrightarrow 2}$ and $\mathcal{C}_{4 \hookrightarrow 2}$ denote the best constants in the embedding $H^1_0(\Omega) \hookrightarrow L^6(\Omega)$ and $\mathbf{H}^1_0(\Omega) \hookrightarrow \mathbf{L}^4(\Omega)$, respectively, and $\mathcal{L}_{\nu}$ and $\mathcal{L}_{\kappa}$ are the Lipschitz constants of $\nu$ and $\kappa$, respectively.
\end{theorem}

\begin{proof}
Let $(\mathbf{u}_2,\mathsf{p}_2,T_2)$ be another solution of the coupled problem \eqref{eq:modelweak}. Define $\mathbf{u} := \mathbf{u}_1-\mathbf{u}_2 \in \mathbf{H}_0^1(\Omega)$, $\mathsf{p} := \mathsf{p}_1 - \mathsf{p}_2 \in L_0^2(\Omega)$, and $T := T_1 -T_2 \in H_0^1(\Omega)$. A simple calculation shows that $T$ verifies the following identity for every $S \in H^1_0(\Omega)$:
\begin{equation}\label{eq:uniqueness_1_heat}
	\int_\Omega \kappa(T_2)\nabla T\cdot \nabla S\mathrm{d}x + \int_\Omega (\mathbf{u}_1\cdot\nabla T_1-\mathbf{u}_2\cdot\nabla T_2)S\mathrm{d}x =\int_\Omega (\kappa(T_2)- \kappa(T_1))\nabla T_1\cdot\nabla S\mathrm{d}x.
\end{equation}
We now set $S=T$ and use that $\int_{\Omega} (\mathbf{u}_2 \cdot  \nabla T) T \mathrm{d}x = 0$ because $\mathbf{u}_2 \in \mathbf{V}(\Omega)$ to obtain
\begin{equation*}
	\int_\Omega \kappa(T_2)|\nabla T|^2 \mathrm{d}x+ \int_\Omega  \left(\mathbf{u} \cdot \nabla T_1 \right) T \mathrm{d}x=\int_\Omega (\kappa(T_2)- \kappa(T_1))\nabla T_1\cdot\nabla T\mathrm{d}x.
\end{equation*}
We use the assumptions on $\kappa$ presented in \S \ref{sec:main_assump},
suitable H\"older's inequalities, the bound $\|\nabla T_1\|_{\mathbf{L}^2(\Omega)}\leq \kappa_{-}^{-1}\|g\|_{H^{-1}(\Omega)}$, and standard Sobolev embeddings to obtain
\begin{multline}
\kappa_{-} \|\nabla T \|^2_{\mathbf{L}^2(\Omega)} \leq \kappa_{-}^{-1}\mathcal{C}^2_{4 \hookrightarrow 2}\| \nabla \mathbf{u} \|_{\mathbf{L}^2(\Omega)}\| g \|_{H^{-1}(\Omega)}  \|\nabla T \|_{\mathbf{L}^2(\Omega)}
\\
+ \mathcal{L}_{\kappa} \mathcal{C}_{6 \hookrightarrow 2}\|\nabla T_1 \|_{\mathbf{L}^3(\Omega)} \|\nabla T \|^2_{\mathbf{L}^2(\Omega)}.
\label{eq:aux_step}
\end{multline}
This estimate leads directly to the following bound:
\begin{equation}\label{eq:bound_T}
	\|\nabla T\|_{\mathbf{L}^2(\Omega)} \leq \left( \kappa_{-}-\mathcal{L}_{\kappa}\mathcal{C}_{6 \hookrightarrow 2}\| \nabla T_1 \|_{\mathbf{L}^3(\Omega)} \right)^{-1}\kappa_{-}^{-1} \mathcal{C}^2_{4 \hookrightarrow 2} \| \nabla \mathbf{u} \|_{\mathbf{L}^2(\Omega)}\| g \|_{H^{-1}(\Omega)}.
\end{equation}
We now apply arguments similar to those we used for \eqref{eq:uniqueness_1_heat} and obtain
\begin{multline}\label{eq:uniqueness_coupled_fluid}
\int_\Omega\Big(\nu(T_2)\nabla\mathbf{u}\cdot\nabla\mathbf{v}+(\mathbf{u}_1\cdot\nabla)\mathbf{u}_1\cdot\mathbf{v}-(\mathbf{u}_2\cdot\nabla)\mathbf{u}_2\cdot\mathbf{v}+ \mathbf{u}\cdot\mathbf{v} \\
+|\mathbf{u}_1|^{s-2}\mathbf{u}_1\cdot\mathbf{v}-|\mathbf{u}_2|^{s-2}\mathbf{u}_2\cdot\mathbf{v}\Big)\mathrm{d}x=\int_\Omega(\nu(T_2)-\nu(T_1))\nabla\mathbf{u}_1\cdot\nabla\mathbf{v}\mathrm{d}x
\end{multline}
for all $\mathbf{v}\in\mathbf{V}(\Omega)$. To control the term $\|\nabla \mathbf{u}\|_{\mathbf{L}^2(\Omega)}$, we set $\mathbf{v}=\mathbf{u}$ in \eqref{eq:uniqueness_coupled_fluid}, invoke \cite[Chapter I, Lemma 4.4]{MR1230384}, and apply the estimates \eqref{eq:estimates_navier} and \eqref{eq:estimate_grad_u} to conclude that
\begin{multline*}
\nu_{-}\|\nabla \mathbf{u}\|_{\mathbf{L}^2(\Omega)}^2
\leq
\nu_{-}^{-1} \mathcal{C}_{N} \|\nabla \mathbf{u}\|_{\mathbf{L}^2(\Omega)}^2  \|\mathbf{f}\|_{\mathbf{H}^{-1}(\Omega)}
\\
+ \mathcal{L}_{\nu} \mathcal{C}_{6 \hookrightarrow 2} \|\nabla T \|_{\mathbf{L}^2(\Omega)}\|\nabla \mathbf{u}_1\|_{\mathbf{L}^3(\Omega)}\|\nabla \mathbf{u} \|_{\mathbf{L}^2(\Omega)}.
\end{multline*}
Replacing \eqref{eq:bound_T} into the previous bound we obtain
\begin{equation*}
\| \nabla \mathbf{u} \|^2_{\mathbf{L}^2(\Omega)} \left(1 - \frac{\mathcal{C}_N}{\nu_{-}^2}\|\mathbf{f}\|_{\mathbf{H}^{-1}(\Omega)} - \frac{\mathcal{L}_{\nu}\mathcal{C}_{6 \hookrightarrow 2}\mathcal{C}^2_{4\hookrightarrow 2}\|g\|_{H^{-1}(\Omega)}\| \nabla \mathbf{u}_1 \|_{\mathbf{L}^3(\Omega)}}{\nu_{-}\kappa_{-}\left( \kappa_{-} - \mathcal{L}_{\kappa}\mathcal{C}_{6 \hookrightarrow 2} \| \nabla T_1 \|_{\mathbf{L}^3(\Omega)} \right)} \right) \leq 0,
\end{equation*}
which, in view of \eqref{eq:small_data_coupled}, immediately shows that $\mathbf{u}=\mathbf{0}$ and therefore $\mathbf{u}_1 = \mathbf{u}_2$. We now substitute $\mathbf{u}=\mathbf{0}$ in \eqref{eq:aux_step} and use assumption \eqref{eq:small_data_coupled_2} to obtain that $T = 0$. Finally, using the inf-sup condition \eqref{eq:infsup_cont}, we arrive at $\mathsf{p} = 0$. This concludes the proof.
\end{proof}

Since different Sobolev embedding results hold in two dimensions, it is possible to improve the regularity assumptions under which it is possible to obtain uniqueness.

\begin{theorem}[uniqueness of solutions]\label{thm:uniq_coupled_2D}
Let the assumptions of Theorem \ref{th:existence_coupled_problem} hold. Let $d=2$ and assume that problem \eqref{eq:modelweak} has a solution $(\mathbf{u}_1,\mathsf{p}_1,T_1)\in\mathbf{W}^{1,2+\varepsilon}(\Omega) \cap \mathbf{H}^1_0 (\Omega)\times L^2_0(\Omega) \times W^{1,2+\varepsilon}(\Omega) \cap H^1_0 (\Omega)$ for some $\epsilon>0$ such that
\begin{align}
\label{eq:small_data_coupled_2D}
\frac{\mathcal{C}_N}{\nu_{-}^2}\|\mathbf{f}\|_{\mathbf{H}^{-1}(\Omega)} + \frac{\mathcal{L}_{\nu}\mathcal{C}_{\varepsilon}\mathcal{C}^2_{4\hookrightarrow 2}\|g\|_{H^{-1}(\Omega)}\| \nabla \mathbf{u}_1 \|_{\mathbf{L}^{2+\varepsilon}(\Omega)}}{\nu_{-}\kappa_{-}\left( \kappa_{-} - \mathcal{L}_{\kappa}\mathcal{C}_{\varepsilon} \| \nabla T_1 \|_{{\mathbf{L}}^{2+\varepsilon}(\Omega)} \right)} & < 1,
\\
\label{eq:small_data_coupled_2D_2}
\mathcal{L}_{\kappa}\mathcal{C}_{\varepsilon} \| \nabla T_1 \|_{{\mathbf{L}}^{2+\varepsilon}(\Omega)} & < \kappa_{-}.
\end{align}
Then, problem (\ref{eq:modelweak}) has no other solution $(\mathbf{u_2},\mathsf{p}_2,T_2)\in\mathbf{H}^1_0(\Omega) \times L^2_0(\Omega) \times H^1_0(\Omega)$. Here, $\mathcal{C}_{\varepsilon}$ and $\mathcal{C}_{4 \hookrightarrow 2}$ denote the best constants in $H^1_0(\Omega) \hookrightarrow L^{2(2+\varepsilon)/\varepsilon}(\Omega)$ and $\mathbf{H}^1_0(\Omega) \hookrightarrow \mathbf{L}^4(\Omega)$, respectively, and $\mathcal{L}_{\nu}$ and $\mathcal{L}_{\kappa}$ are the Lipschitz constants of $\nu$ and $\kappa$, respectively.
\end{theorem}


\section{Finite element approximation}
\label{sec:fem}
In this section, we propose and analyze a finite element scheme to approximate solutions of the nonlinear coupled problem \eqref{eq:modelweak}. As far as the analysis is concerned, we study the convergence properties of the discretization scheme and derive a priori error bounds. To this end, we first introduce some notions and basic ingredients \cite{MR3097958,MR2373954,CiarletBook,Guermond-Ern}.

\subsection{Notation and basic components}
From now on we assume that $\Omega$ is a Lipschitz polytope so that it can be exactly triangulated. Let $\mathscr{T}_h = \{K\}$ be a conforming partition of $\bar{\Omega}$ into closed simplices $K$ of size $h_K = \operatorname*{diam}(K)$. Here, $h:= \max \{ h_K : K \in \mathscr{T}_h \}$. We denote by $\mathbb{T}=\{\mathscr{T}_h\}_{h>0}$ a collection of conforming meshes $\mathscr{T}_h$, which are refinements of an initial mesh $\mathscr{T}_0$. We assume that the collection $\mathbb{T}$ satisfies the so-called shape regularity condition.

\subsection{Finite element spaces}
Let $\mathscr{T}_h$ be a mesh in $\mathbb{T}$. To approximate the velocity field $\mathbf{u}$ and the pressure $\mathsf{p}$ of the fluid, we consider a pair $(\mathbf{X}_h,M_h)$ of finite element spaces satisfying a uniform discrete inf-sup condition: There exists a constant $\tilde{\beta}>0$ independent of $h$ such that
\begin{equation}\label{eq:discrete_infsup}
\inf _{\mathsf{q}_h \in {M}_h} \sup _{\mathbf{v}_h \in \mathbf{X}_h} \frac{ \int_{\Omega} \mathsf{q}_h\operatorname{div} \mathbf{v}_h \mathrm{d}x }{\left\|\nabla\mathbf{v}_h\right\|_{\mathbf{L}^2(\Omega)}\left\|\mathsf{q}_h\right\|_{L^2(\Omega)}} \geqslant \tilde{\beta}>0.
\end{equation}
We will look in particular at the following pairs, which are significant in the literature:
\begin{itemize}[leftmargin=*]
\item[(1)] The \emph{lowest order Taylor--Hood element} introduced in \cite{MR0339677} for $d=2$; see also \cite[Chapter II, Section 4.2]{MR851383}, \cite[Section 4.2.5]{Guermond-Ern}, \cite[Section 8.8.2]{MR3097958}: In this case,
\begin{align}
\label{TH:vel_space}
\mathbf{X}_h & = \{\mathbf{v}_h \in\mathbf{C}(\bar{\Omega}): \mathbf{v}_h|_K \in\left[\mathbb{P}_2(K)\right]^d
~\forall K \in \mathscr{T}_h \}\cap \mathbf{H}_0^1(\Omega),
\\
\label{TH:press_space}
M_h & = \{\mathsf{q}_h \in L_{0}^2(\Omega) \cap C(\bar{\Omega}) : \mathsf{q}_h|_K \in \mathbb{P}_1(K)
~\forall K \in \mathscr{T}_h\}.
\end{align}
\item[(2)] The \emph{mini element} introduced in \cite{MR0799997} for $d=2$; see also \cite[Chapter II, Section 4.1]{MR851383}, \cite[Section 4.2.4]{Guermond-Ern}, \cite[Section 8.4.2]{MR3097958}: In this scenario,
\begin{align}
\label{ME:vel_space}
\mathbf{X}_h
&
=\{\mathbf{v}_{h}\in\mathbf{C}(\bar{\Omega}):\mathbf{v}_{h}|_{K}\in[\mathbb{W}(K)]^{d}
~\forall K\in\T_h \} \cap \mathbf{H}_0^1(\Omega),
\\
\label{ME:press_space}
M_h &
=
\{ \mathsf{q}_{h}\in L_0^2(\Omega)\cap C(\bar{\Omega}): \mathsf{q}_{h}|_{K}\in\mathbb{P}_1(K)
~\forall K\in\T_h \},
\end{align}
where $\mathbb{W}(K):=\mathbb{P}_1(K)\oplus \mathbb{B}(K)$, and $\mathbb{B}(K)$ denotes the space spanned by a local bubble function.
\end{itemize}

A proof of the inf-sup condition \eqref{eq:discrete_infsup} for the mini element can be found in \cite[Lemma 4.20]{Guermond-Ern}. Provided that the mesh $\T_h$ contains at least three triangles in two dimensions and that each tetrahedron has at least one internal vertex in three dimensions, a proof of \eqref{eq:discrete_infsup} for the Taylor--Hood element can be found in \cite[Theorem 8.8.1]{MR3097958} and \cite[Theorem 8.8.2]{MR3097958}, respectively.

To approximate the temperature variable of the fluid, we consider the space
\begin{equation}\label{eq:discrete_temp_space}
	Y_h:=\left\{S_h \in C(\bar{\Omega}):S_h|_K \in \mathbb{P}_{\mathfrak{r}}(K)~\forall K \in \mathscr{T}_h\right\}\cap H_0^1(\Omega),
\end{equation}
where $\mathfrak{r}=1$ if the mini element is used to approximate $(\mathbf{u},\mathsf{p})$, or $\mathfrak{r}=2$ if we approximate $(\mathbf{u},\mathsf{p})$ with the Taylor--Hood element.

The finite element spaces $\mathbf{X}_h$ and $Y_h$ satisfy the following basic approximation properties: For all $\mathbf{v} \in \mathbf{H}_0^1(\Omega)$ and $S \in H_0^1(\Omega)$, we have
\begin{equation}
\label{eq:density_1}
\lim_{h \rightarrow 0} \left( \inf_{\mathbf{v}_h \in \mathbf{X}_h} \| \nabla(\mathbf{v} - \mathbf{v}_h)\|_{\mathbf{L}^2(\Omega)} \right)  = 0,
\qquad
\lim_{h \rightarrow 0} \left( \inf_{S_h \in Y_h} \| \nabla(S - S_h)\|_{\mathbf{L}^2(\Omega)} \right)  = 0.
 \end{equation}
We also have the existence of interpolation operators $\mathbf{\Pi}_h: \mathbf{H}_0^1(\Omega) \rightarrow \mathbf{X}_h$ and $\Pi_h: H_0^1(\Omega) \rightarrow Y_h$, so that, for every $\mathbf{v} \in \mathbf{H}_0^1(\Omega)$ and $S \in H_0^1(\Omega)$, we have
\begin{equation}
\lim_{h \rightarrow 0}
\| \nabla(\mathbf{v} - \mathbf{\Pi}_h \mathbf{v})\|_{\mathbf{L}^2(\Omega)} \rightarrow 0,
 \qquad
\lim_{h \rightarrow 0}
\| \nabla(S - \Pi_h S) \|_{\mathbf{L}^2(\Omega)} \rightarrow 0.
\label{eq:density_2}
\end{equation}

Finally, we introduce the space $\mathbf{V}_h:= \{ \mathbf{v}_h \in \mathbf{X}_h : \int_{\Omega} \mathsf{q}_h \operatorname*{div}\mathbf{v}_h \mathrm{d}x = 0 \, \forall \mathsf{q}_h \in {M}_h \}$.

\subsection{The discrete coupled problem}

Before presenting a discrete scheme for the nonlinear problem \eqref{eq:modelweak}, we introduce some basic ingredients. First, we define
\begin{align*}\label{eq:discrete_nonlinear}
& \mathfrak{a}: Y_h \times \mathbf{X}_h \times \mathbf{X}_h\to\mathbb{R},
\qquad
\mathfrak{a}(T_h;\mathbf{u}_h,\mathbf{v}_h):=\int_{\Omega} \left(\nu(T_h) \nabla\mathbf{u}_h\cdot\nabla\mathbf{v}_h+\mathbf{u}_h\cdot\mathbf{v}_h\right)\mathrm{d}x,
\\
& \mathfrak{b}: Y_h \times Y_h \times Y_h \to \mathbb{R},
\qquad
\mathfrak{b}(R_h;T_h,S_h):=\int_{\Omega} \kappa(R_h) \nabla T_h\cdot\nabla S_h \mathrm{d}x.
\end{align*}
Since functions in $\mathbf{V}_{h}$ are generally not divergence-free, we follow \cite[Chapter II, \S 3.2]{MR0603444} and \cite[Chapter IV, \S 3]{MR0548867} and introduce a slight variation of the form $a_N$ that preserves the antisymmetry on a discrete level. To be precise, we define  $a_{SN}:\mathbf{X}_h^3 \rightarrow \mathbb{R}$ by
\begin{align}\label{eq:discrete_SN}
	a_{SN}(\mathbf{u}_h;\mathbf{v}_h,\mathbf{w}_h):= \frac{1}{2}\left[a_N(\mathbf{u}_h;\mathbf{v}_h,\mathbf{w}_h)-a_N(\mathbf{u}_h;\mathbf{w}_h,\mathbf{v}_h)\right].
\end{align}
Note that $a_{SN}(\mathbf{u}_h;\mathbf{v}_h,\mathbf{v}_h) = 0$ for $\mathbf{u}_h, \mathbf{v}_h \in \mathbf{X}_h$. Similarly, we introduce the form $a_{ST}: \mathbf{X}_h \times Y_h \times Y_h \rightarrow \mathbb{R}$ by
\begin{equation}\label{eq:discrete_ST}
a_{ST}(\mathbf{v}_h;T_h,S_h):= \frac{1}{2}\int_\Omega \left((\mathbf{v}_h\cdot\nabla T_h)S_h - (\mathbf{v}_h\cdot\nabla S_h)T_h\right)\mathrm{d}x.
\end{equation}
Note that $a_{ST}(\mathbf{v}_h;S_h,S_h) = 0$ for $\mathbf{v}_h \in \mathbf{X}_h$ and $S_h \in Y_h$.

With all these ingredients, we introduce the following discrete approximation of \eqref{eq:modelweak}: Given $\mathbf{f}\in\mathbf{H}^{-1}(\Omega)$ and $g\in H^{-1}(\Omega)$, find $(\mathbf{u}_h,\mathsf{p}_h,T_h) \in \mathbf{X}_h \times M_h \times Y_h$ such that
\begin{multline}
\label{eq:modelweak_discrete}
\mathfrak{a}(T_h;\mathbf{u}_h, \mathbf{v}_h)
+
a_{SN}(\mathbf{u}_h;\mathbf{u}_h,\mathbf{v}_h)
+
a_{F}(\mathbf{u}_h;\mathbf{u}_h,\mathbf{v}_h)
+
b(\mathbf{v}_h,\mathsf{p}_h)=\langle \mathbf{f}, \mathbf{v}_h \rangle,
\\
b(\mathbf{u}_h,\mathsf{q}_h)=0,
\qquad
\mathfrak{b}(T_h;T_h, S_h)+a_{ST}(\mathbf{u}_h;T_h,S_h)=\langle g,S_h\rangle,
\end{multline}
for all $(\mathbf{v}_h,\mathsf{q}_h,S_h) \in \mathbf{X}_h \times M_h \times Y_h$ and $s \in [3,4]$.

In the following we prove that for every $h>0$ the discrete problem \eqref{eq:modelweak_discrete} has a solution and that the sequence of solutions $\{ (\mathbf{u}_h, \mathsf{p}_h, T_h )\}_{h>0}$ is uniformly bounded in $\mathbf{H}_0^1(\Omega) \times L_0^2(\Omega) \times H_0^1(\Omega)$.

\begin{theorem}[existence of solutions]\label{th:existence_coupled_problem_discrete}
In the framework of Theorem \ref{th:existence_coupled_problem}, there is at least one solution $(\mathbf{u}_h,\mathsf{p}_h,T_h) \in \mathbf{X}_h \times M_h \times Y_h$ for problem \eqref{eq:modelweak_discrete}. Moreover,
\begin{align}
\label{eq:estimate_grad_u_h}
\|\nabla \mathbf{u}_h\|_{\mathbf{L}^2(\Omega)}
& \leq \nu_{-}^{-1} \|\mathbf{f}\|_{\mathbf{H}^{-1}(\Omega)},
\quad
\| \mathsf{p}_h \|_{L^2(\Omega)}
\leq
\tilde{\beta}^{-1} \Lambda(\mathbf{f})\|\mathbf{f}\|_{\mathbf{H}^{-1}(\Omega)},
\\
\label{eq:estimate_T_h}
\|\nabla T_h \|_{\mathbf{L}^2(\Omega)}
& \leq \kappa_{-}^{-1} \| g \|_{H^{-1}(\Omega)},
\end{align}
where $\Lambda$ is defined in \eqref{eq:Lambda_f}.
\end{theorem}
\begin{proof}
As in the proof of Theorem \ref{theorem_coupled_existence}, we define a suitable operator $\Phi_h$ from $\mathbf{V}_h \times Y_h$ into itself. Since $a_{SN}(\mathbf{u}_h;\mathbf{v}_h,\mathbf{v}_h) = 0$ for $\mathbf{u}_h, \mathbf{v}_h \in \mathbf{X}_h$ and $a_{ST}(\mathbf{v}_h;S_h,S_h) = 0$ for $\mathbf{v}_h \in \mathbf{X}_h$ and $S_h \in Y_h$, it can be proved that $\langle \Phi(\boldsymbol{\mathcal{V}}_h),\boldsymbol{\mathcal{V}}_h \rangle \geq 0$ for any $\boldsymbol{\mathcal{V}}_h = (\mathbf{v}_h,S_h) \in \mathbf{V}_h \times Y_h$ such that $\| (\mathbf{v}_h,S_h) \|_{\mathbf{H}_0^1(\Omega) \times H_0^1(\Omega)} = \delta$. The existence of a solution $(\mathbf{u}_h,\mathsf{p}_h,T_h) \in \mathbf{X}_h \times M_h \times Y_h$ therefore follows directly from \cite[Chapter IV, Corollary 1.1]{MR851383} and the discrete inf-sup condition \eqref{eq:discrete_infsup}. Finally, the discrete stability bounds \eqref{eq:estimate_grad_u_h}--\eqref{eq:estimate_T_h} follow from the arguments developed in the proof of Theorem \ref{th:existence_coupled_problem} (see step 5).
\end{proof}

As for the continuous case, we can also show the uniqueness of solutions for the discrete problem \eqref{eq:modelweak_discrete} under suitable assumptions.

\begin{theorem}[uniqueness of solutions]\label{thm:uniq_coupled_3D_discrete}
Let the assumptions of Theorem \ref{theorem_coupled_existence} hold. Let $d=3$ and assume that problem \eqref{eq:modelweak_discrete} has a solution ${(\mathbf{u}_h}_1,{\mathsf{p}_h}_1,{T_h}_1)\in\mathbf{W}^{1,3}(\Omega) \cap \mathbf{X}_h \times M_h \times W^{1,3}(\Omega) \cap Y_h$ such that
		\begin{align}
			\label{eq:small_data_coupled_discrete}
			\frac{\mathcal{C}_N}{\nu_{-}^2}\|\mathbf{f}\|_{\mathbf{H}^{-1}(\Omega)} + \frac{\mathcal{L}_{\nu}\mathcal{C}_{6 \hookrightarrow 2}\mathcal{C}^2_{4\hookrightarrow 2}\|g\|_{H^{-1}(\Omega)}\| \nabla {\mathbf{u}_h}_1 \|_{\mathbf{L}^3(\Omega)}}{\nu_{-}\kappa_{-}\left( \kappa_{-} - \mathcal{L}_{\kappa}\mathcal{C}_{6 \hookrightarrow 2} \| \nabla {T_h}_1 \|_{{\mathbf{L}}^3(\Omega)} \right)}  & < 1,
			\\
			\label{eq:small_data_coupled_discrete_2}
			\mathcal{L}_{\kappa}\mathcal{C}_{6 \hookrightarrow 2}\| \nabla {T_h}_1 \|_{{\mathbf{L}}^3(\Omega)}  & <  \kappa_{-}.
		\end{align}
		Then, \eqref{eq:modelweak_discrete} has no other solution $({\mathbf{u}_h}_2,{\mathsf{p}_h}_2,{T_h}_2) \in \mathbf{X}_h \times M_h \times Y_h$. Here, $\mathcal{C}_{6 \hookrightarrow 2}$ and $\mathcal{C}_{4 \hookrightarrow 2}$ denote the best constants in the embedding $H^1_0(\Omega) \hookrightarrow L^6(\Omega)$ and $\mathbf{H}^1_0(\Omega) \hookrightarrow \mathbf{L}^4(\Omega)$, respectively, and $\mathcal{L}_{\nu}$ and $\mathcal{L}_{\kappa}$ are the Lipschitz constants of $\nu$ and $\kappa$, respectively.
\end{theorem}
\begin{proof}
Let $({\mathbf{u}_h}_2,{\mathsf{p}_h}_2,{T_h}_2) \in \mathbf{X}_h \times M_h \times Y_h$ be another solution to problem \eqref{eq:modelweak_discrete} and define $\mathbf{u}_h := {\mathbf{u}_h}_1-{\mathbf{u}_h}_2 \in \mathbf{X}_h$, $\mathsf{p}_h := {\mathsf{p}_h}_1 - {\mathsf{p}_h}_2 \in M_h$, and $T_h := {T_h}_1 -{T_h}_2 \in Y_h$. We proceed similarly to the proof of Theorem \ref{thm:uniq_coupled_3D} with special considerations on the bilinear forms $a_{SN}$ and $a_{ST}$. This concludes the proof.
\end{proof}

As in the continuous case, the regularity requirements for the solution in dimension 2 can be relaxed.

\begin{theorem}[uniqueness of solutions]\label{thm:uniq_coupled_2D_discrete}
Let the assumptions of Theorem \ref{theorem_coupled_existence} hold. Let $d=2$ and assume that problem \eqref{eq:modelweak_discrete} has a solution ${(\mathbf{u}_h}_1,{\mathsf{p}_h}_1,{T_h}_1)\in\mathbf{W}^{1,2+\varepsilon}(\Omega) \cap \mathbf{X}_h \times M_h \times W^{1,2+\varepsilon}(\Omega) \cap Y_h$ for some $\varepsilon > 0$ such that
\begin{align}
	\label{eq:small_data_coupled_discrete_2D}
	\frac{\mathcal{C}_N}{\nu_{-}^2}\|\mathbf{f}\|_{\mathbf{H}^{-1}(\Omega)} + \frac{\mathcal{L}_{\nu}\mathcal{C}_{\varepsilon}\mathcal{C}^2_{4\hookrightarrow 2}\|g\|_{H^{-1}(\Omega)}\| \nabla {\mathbf{u}_h}_1 \|_{\mathbf{L}^{2+\varepsilon}(\Omega)}}{\nu_{-}\kappa_{-}\left( \kappa_{-} - \mathcal{L}_{\kappa}\mathcal{C}_{\varepsilon} \| \nabla {T_h}_1 \|_{{\mathbf{L}}^{2+\varepsilon}(\Omega)} \right)}  & < 1,
	\\
	\label{eq:small_data_coupled_discrete_2_2D}
	\mathcal{L}_{\kappa}\mathcal{C}_{\varepsilon}\| \nabla {T_h}_1 \|_{{\mathbf{L}}^{2+\varepsilon}(\Omega)}  & <  \kappa_{-}.
\end{align}
Then, \eqref{eq:modelweak_discrete} has no other solution $({\mathbf{u}_h}_2,{\mathsf{p}_h}_2,{T_h}_2)\in \mathbf{X}_h \times M_h \times Y_h$. Here, $\mathcal{C}_{\varepsilon}$ and $\mathcal{C}_{4 \hookrightarrow 2}$ denote the best constants in the embeddings $H^1_0(\Omega) \hookrightarrow L^{2(2+\varepsilon)/\varepsilon}(\Omega)$ and $\mathbf{H}^1_0(\Omega) \hookrightarrow \mathbf{L}^4(\Omega)$, respectively, and $\mathcal{L}_{\nu}$ and $\mathcal{L}_{\kappa}$ are the Lipschitz constants of $\nu$ and $\kappa$, respectively.
\end{theorem}

\subsection{Convergence}
We present the following convergence result. The most important feature of this result is that \emph{it requires no regularity properties of solutions other than those required for the well-posedness of the problem, no additional smallness assumptions, and no further regularity properties of $\Omega$ beyond a Lipschitz property.}

\begin{theorem}[convergence result]\label{thm:convergence_coupled_problem}
In the framework of Theorem \ref{th:existence_coupled_problem}, let $h > 0$, and let $(\mathbf{u}_h,\mathsf{p}_h,T_h)$ be a solution of the discrete problem \eqref{eq:modelweak_discrete}. Then, there exists a nonrelabeled subsequence of $\{(\mathbf{u}_h,\mathsf{p}_h,T_h)\}_{h>0}$ such that
$
 \mathbf{u}_h \rightharpoonup \mathbf{u},
$
$
\mathsf{p}_h \rightharpoonup \mathsf{p},
$
and
$T_h \rightharpoonup T
$
in $\mathbf{H}^1_0(\Omega)$, $L^2_0(\Omega)$, and $H^1_0(\Omega)$, respectively, as $h \rightarrow 0$. Moreover, $(\mathbf{u},\mathsf{p},T)$ solves \eqref{eq:modelweak}.
\end{theorem}
\begin{proof}
For $h>0$, the existence of a discrete solution $(\mathbf{u}_h,\mathsf{p}_h,T_h) \in \mathbf{X}_h \times M_h \times Y_h$ is guaranteed by Theorem \ref{th:existence_coupled_problem_discrete}. We note that, given the stability bounds in Theorem \ref{th:existence_coupled_problem_discrete}, the sequence $\{ (\mathbf{u}_h, \mathsf{p}_h, T_h) \}_{h>0}$ is uniformly bounded in $\mathbf{H}^1_0(\Omega) \times L^2_0(\Omega) \times H^1_0(\Omega)$ with respect to the discretization parameter $h$. We can therefore deduce the existence of a nonrelabeled subsequence $\{ (\mathbf{u}_h,\mathsf{p}_{h},T_{h}) \}_{h>0}$ such that
\begin{equation}
 (\mathbf{u}_h,\mathsf{p}_{h},T_{h}) \rightharpoonup (\mathbf{u},\mathsf{p},T), \qquad h \rightarrow 0,
\end{equation}
in $ \mathbf{H}_0^1(\Omega) \times L^2(\Omega) \times H^1_0(\Omega)$. In the following, we show that $(\mathbf{u},\mathsf{p},T)\in \mathbf{H}_0^1(\Omega) \times L^2_0(\Omega) \times H^1_0(\Omega)$ solves the nonlinear problem \eqref{eq:modelweak}. To do this, we proceed in several steps.

\emph{Step 1.} We show that the pair $(\mathbf{u},\mathsf{p}) \in \mathbf{H}_0^1(\Omega)\times L_0^2(\Omega)$ solves the first two equations in \eqref{eq:modelweak} with temperature $T$. To do this, we let $\mathbf{v} \in \mathbf{H}_{0}^{1}(\Omega)$ be arbitrary and set $\mathbf{v}_h = \mathbf{\Pi}_h \mathbf{v} \in \mathbf{X}_{h}$. A simple calculation for the Forchheimer term shows that
	\begin{multline*}
			\left|
			\int_{\Omega}\left( |\mathbf{u}_h|^{s-2}\mathbf{u}_h\cdot\mathbf{v}_h-|\mathbf{u}|^{s-2}\mathbf{u}\cdot\mathbf{v} \right)\mathrm{d}x \right| \leq\left|\int_{\Omega} |\mathbf{u}_h |^{s-2}\mathbf{u}_h \cdot\left(\mathbf{v}_h-\mathbf{v}\right)\mathrm{d}x
			\right|
			\\
			\quad +
			\left|
			\int_{\Omega} \left(  |\mathbf{u}_h|^{s-2}\mathbf{u}_h-|\mathbf{u}|^{s-2}\mathbf{u} \right)\cdot\mathbf{v} \mathrm{d}x
			\right|=:\text{I}_h +\text{II}_h.
	\end{multline*}
Given \eqref{eq:estimates_a_weakly_cont} and the fact that the embedding $\mathbf{H}^1_0(\Omega) \hookrightarrow \mathbf{L}^{\iota}(\Omega)$ is compact for $\iota < \infty$ when $d=2$ and $\iota < 6$ when $d=3$, we can conclude that $\text{II}_h\to 0$ as $h \rightarrow 0$. To analyze the term $\text{I}_h$, we first invoke \eqref{eq:estimate_for_Forchheimer_2} to obtain
\begin{equation*}
\left|\int_{\Omega} |\mathbf{u}_h |^{s-2}\mathbf{u}_h \cdot\left(\mathbf{v}_h-\mathbf{v}\right) \mathrm{d}x \right|
\lesssim  \|\nabla (\mathbf{v} - \mathbf{v}_h)\|_{\mathbf{L}^2(\Omega)} \| \nabla \mathbf{u}_h \|^{s-1}_{\mathbf{L}^2(\Omega)}.
\end{equation*}
Since $\{ \mathbf{u}_h \}_{h >0}$ is uniformly bounded in $\mathbf{H}_0^1(\Omega)$, a result in \eqref{eq:density_2} allows us to conclude.

In the following, we
prove that $a_{SN}(\mathbf{u}_h;\mathbf{u}_h,\mathbf{v}_h) \rightarrow \int_\Omega (\mathbf{u} \cdot \nabla )\mathbf{u}\cdot\mathbf{v}\mathrm{d}x$ as $h \rightarrow 0$. To do so, we use the definition of $a_{SN}$ from \eqref{eq:discrete_SN} and then a property from \eqref{eq:properties_navier} to obtain
\begin{multline*}
\left|\int_{\Omega}(\mathbf{u} \cdot \nabla)\mathbf{u}\cdot\mathbf{v}\mathrm{d}x
-
a_{SN}(\mathbf{u}_h;\mathbf{u}_h,\mathbf{v}_h)\right|
\leq
\frac{1}{2}
\left|\int_{\Omega}\left((\mathbf{u} \cdot \nabla )\mathbf{u}\cdot\mathbf{v}-(\mathbf{u}_h \cdot \nabla )\mathbf{u}_h\cdot\mathbf{v}_h\right) \mathrm{d}x\right|
\\
+
\frac{1}{2}
\left|\int_{\Omega}\left((\mathbf{u}_h \cdot \nabla )\mathbf{v}_h\cdot \mathbf{u}_h-(\mathbf{u} \cdot \nabla )\mathbf{v}\cdot \mathbf{u}\right) \mathrm{d}x\right| =: \text{III}_h + \text{IV}_h.
\end{multline*}
To control $\text{III}_h$, we use $\mathbf{u}_h \rightharpoonup \mathbf{u}$ in $\mathbf{H}^1_0(\Omega)$ as $k \uparrow \infty$, the fact that $\mathbf{u}, \mathbf{v} \in \mathbf{L}^4(\Omega)$, the compact embedding $\mathbf{H}_0^1(\Omega) \hookrightarrow \mathbf{L}^q(\Omega)$ for $q<\infty$ when $d=2$ and $q<6$ when $d=3$, the fact that $\{ \mathbf{u}_h \}_{h>0}$ is uniformly bounded in $\mathbf{H}_0^1(\Omega)$, and \eqref{eq:density_2} to obtain
\begin{multline*}
	2 \text{III}_{h}
	\leq
	\left|\int_{\Omega}\left(\mathbf{u} \cdot \nabla\right)(\mathbf{u}- \mathbf{u}_h)\cdot\mathbf{v} \mathrm{d}x\right|
	+
	\left|\int_{\Omega} ((\mathbf{u} - \mathbf{u}_h) \cdot \nabla) \mathbf{u}_h\cdot\mathbf{v}\mathrm{d}x \right|
	\\
	+\left|\int_{\Omega} (\mathbf{u}_h \cdot \nabla)\mathbf{u}_h\cdot(\mathbf{v}-\mathbf{v}_h)\mathrm{d}x \right| \rightarrow 0,
	\qquad h \rightarrow 0.
\end{multline*}
We apply similar arguments to control the term $\text{IV}_{h}$:
\begin{multline*}
2\text{IV}_{h} \leq\left|\int_{\Omega}\left((\mathbf{u}-\mathbf{u}_h) \cdot \nabla\right)\mathbf{v}\cdot\mathbf{u}\mathrm{d}x\right|
+\left|\int_{\Omega} (\mathbf{u}_h \cdot \nabla) \mathbf{v}\cdot(\mathbf{u}-\mathbf{u}_h) \mathrm{d}x\right|
\\
+
\left|\int_{\Omega} (\mathbf{u}_h \cdot \nabla) (\mathbf{v}-\mathbf{v}_h)\cdot\mathbf{u}_h \mathrm{d}x\right| \rightarrow 0,
\qquad
h \rightarrow 0.
\end{multline*}

The convergence of the linear term is trivial: $|\int_{\Omega}(\mathbf{u}\cdot\mathbf{v} - \mathbf{u}_h\cdot\mathbf{v}_h)\mathrm{d}x|\rightarrow 0$ as $h \rightarrow 0$.

We now show that $( \nu(T_h) \nabla \mathbf{u}_h, \nabla \mathbf{v}_h)_{\mathbf{L}^2(\Omega)} \rightarrow ( \nu(T) \nabla \mathbf{u}, \nabla \mathbf{v})_{\mathbf{L}^2(\Omega)}$ as $h \rightarrow 0$. To achieve this, we proceed as follows. First, as in step 3 of the proof of Theorem \ref{th:existence_coupled_problem}, we have that $\nu(T_h) \nabla \mathbf{v} \rightarrow \nu(T) \nabla \mathbf{v}$ in $\mathbf{L}^2(\Omega)$ as $h \rightarrow 0$. This and the weak convergence $\mathbf{u}_h \rightharpoonup \mathbf{u}$ in $\mathbf{H}_0^1(\Omega)$ show that $(\nu(T_h) \nabla \mathbf{u}_h, \nabla \mathbf{v})_{\mathbf{L}^2(\Omega)} \rightarrow ( \nu(T) \nabla \mathbf{u}, \nabla \mathbf{v} )_{\mathbf{L}^2(\Omega)}$ as $h \rightarrow 0$. The fact that $| (\nu(T_h) \nabla \mathbf{u}_h, \nabla (\mathbf{v} - \mathbf{v}_h ))_{\mathbf{L}^2(\Omega)}| \rightarrow 0$ as $h \rightarrow 0$ follows from the boundedness of $\nu$, the uniform boundedness of $\{ \mathbf{u}_h \}_{h > 0}$ in $\mathbf{H}_0^1(\Omega)$ and \eqref{eq:density_2}.

Since $\mathsf{p}_h \rightharpoonup \mathsf{p}$ in $L^2(\Omega)$ as $h \rightarrow 0$, the fact that $\{ \mathsf{p}_h \}_{h>0}$ is uniformly bounded in $L^2(\Omega)$, the bound $\| \operatorname*{div}\mathbf{w} \|_{L^2(\Omega)} \leq \| \nabla \mathbf{w} \|_{\mathbf{L}^2(\Omega)}$, which holds for every $\mathbf{w} \in \mathbf{H}_0^1(\Omega)$, and a result from \eqref{eq:density_2} lead us to the following conclusions:
	\begin{multline*}
		\left|\int_{\Omega}\left(\mathsf{p} \operatorname*{div}\mathbf{v} - \mathsf{p}_h \operatorname*{div}\mathbf{v}_h \right)\mathrm{d}x\right|
		\leq
		\left|\int_{\Omega}(\mathsf{p} - \mathsf{p}_h) \operatorname*{div}\mathbf{v} \mathrm{d}x\right| \\
		+\|\mathsf{p}_h\|_{L^2(\Omega)}\| \nabla (\mathbf{v}-\mathbf{v}_h)\|_{\mathbf{L}^2(\Omega)} \rightarrow 0,
		\qquad
		h \rightarrow 0.
	\end{multline*}

Finally, the fact that $\int_{\Omega}\mathsf{q}\,\mathrm{div }\mathbf{u}\,\mathrm{d}x = 0$ for $\mathsf{q} \in L_0^2(\Omega)$ is trivial.

\emph{Step 3}. We prove that $T$ solves the heat equation in  \eqref{eq:modelweak}. For this purpose, we let $S\in H_0^{1}(\Omega)$ be arbitrary and set $S_h = \Pi_h S \in Y_h$. Next, we write the bound
\begin{equation*}
\begin{aligned}
\mathfrak{I}_h
& :=
\left|
\int_{\Omega}\left(\kappa(T) \nabla T \cdot \nabla S-\kappa\left(T_h\right) \nabla T_h \cdot \nabla S_h\right)\mathrm{d}x
\right|
\\
& \leq
\left|
\int_{\Omega}\left(\kappa(T) \nabla T -\kappa(T_h) \nabla T_h \right)  \cdot \nabla S\mathrm{d}x
\right|
+
\left|\int_{\Omega} \kappa\left(T_h\right) \nabla T_h \cdot \nabla\left(S-S_h\right) \mathrm{d}x\right|.
\end{aligned}
\end{equation*}
The strong convergence $\kappa(T_h) \nabla S \rightarrow \kappa(T) \nabla S$ in $\mathbf{L}^2(\Omega)$ as $h \rightarrow 0$, the weak convergence $T_h \rightharpoonup T$ in $H^1_0(\Omega)$ as $h \rightarrow 0$, the boundedness of $\kappa$, the uniform boundedness of $\{ T_h \}_{h > 0}$ in $H_0^1(\Omega)$, and a result from \eqref{eq:density_2} show that $\mathfrak{I}_h \rightarrow 0$ as $h \rightarrow 0$.

We now prove that $ a_{ST}(\mathbf{u}_h;T_h,S_h) \rightarrow \int_\Omega (\mathbf{u} \cdot \nabla T)S\mathrm{d}x$ as $h \rightarrow 0$. Since $\operatorname*{div}\mathbf{u}=0$ and $T,S \in H^1_0(\Omega)$, a simple calculation shows that
\begin{multline*}
\left|\int_\Omega (\mathbf{u} \cdot \nabla T)S\mathrm{d}x-a_{ST}(\mathbf{u}_h;T_h,S_h)\right|
\leq
\frac{1}{2}
\left|\int_{\Omega}\left((\mathbf{u} \cdot \nabla T)S-(\mathbf{u}_h \cdot \nabla T_h)S_h\right) \mathrm{d}x\right|
\\
+
\frac{1}{2}
\left|\int_{\Omega}\left((\mathbf{u}_h \cdot \nabla S_h)T_h-(\mathbf{u} \cdot \nabla S)T\right) \mathrm{d}x\right| =: \text{V}_{h} + \text{VI}_{h}.
\end{multline*}
To control $\text{V}_h$, we use the strong convergence $\mathbf{u}_h \to \mathbf{u}$ in $\mathbf{L}^4(\Omega)$, the weak convergence $T_h \rightharpoonup T$ in $H^1_0(\Omega)$ as $h \rightarrow 0$, the uniform boundedness of $\{ T_h \}_{h>0}$ and $\{ \mathbf{u}_h \}_{h>0}$ in $H_0^1(\Omega)$ and $\mathbf{H}_0^1(\Omega)$, respectively, and \eqref{eq:density_2}. These arguments show that
\begin{multline*}
2\text{V}_{h}
\leq
\left|\int_{\Omega}\left( (\mathbf{u} -\mathbf{u}_h) \cdot \nabla T\right)S\mathrm{d}x \right|
+
\left|\int_{\Omega} (\mathbf{u}_h \cdot \nabla (T-T_h))S \mathrm{d}x\right|
\\
+
\left|\int_{\Omega} (\mathbf{u}_h \cdot \nabla T_h)(S - S_h) \mathrm{d}x \right| \rightarrow 0,
\qquad
h \rightarrow 0.
\end{multline*}
We apply similar arguments to control the term $\text{VI}_h$:
\begin{multline*}
2\text{VI}_{h}
\leq
\left|\int_{\Omega}\left((\mathbf{u} - \mathbf{u}_h) \cdot \nabla S \right)T \mathrm{d}x\right|
+
\left|\int_{\Omega} (\mathbf{u}_h \cdot \nabla (S-S_h))T \mathrm{d}x\right|
\\
+\left|\int_{\Omega} (\mathbf{u}_h \cdot \nabla S_h)(T - T_h) \mathrm{d}x\right| \rightarrow 0,
\qquad
h \rightarrow 0.
\end{multline*}

We have thus proved that $T$ solves the heat equation in \eqref{eq:modelweak}. With the previous results, we can conclude that $(\mathbf{u},\mathsf{p},T)$ solves \eqref{eq:modelweak}. This concludes the proof.
\end{proof}

\subsection{A quasi-best approximation result}
We derive a quasi-best approximation result for the finite element approximation \eqref{eq:modelweak_discrete} of problem \eqref{eq:modelweak}. For this purpose, we assume suitable smallness conditions and regularity assumptions for the solution; see Theorems \ref{thm:apriori_error_3D} and \ref{thm:apriori_error_2D} below. Since the regularity requirements on the solution in two dimensions are weaker, we first derive the error bounds in three dimensions and write the corresponding error bounds in two dimensions in a separate theorem.

We begin our analysis by introducing the errors $\mathbf{e}_{\mathbf{u}}:=\mathbf{u}-\mathbf{u}_h$, $e_{\mathsf{p}}:=\mathsf{p}-\mathsf{p}_h$, and $e_{T}:=T-T_h$ as well as the quantities
	\begin{align}
		\label{eq:apriori_3D_constants_1}
		\mathfrak{M}(T)
		& := \kappa_{-} - \mathcal{L}_{\kappa}\mathcal{C}_{6 \hookrightarrow 2}\| \nabla T \|_{{\mathbf{L}}^3(\Omega)},
		\\
		\label{eq:apriori_3D_constants_2}
		\mathfrak{N}(\mathbf{f},g,\mathbf{u})
		& :=
		\nu_{-}
		-
		\nu_{-}^{-1}\mathcal{C}_N\|\mathbf{f}\|_{\mathbf{H}^{-1}(\Omega)}
		- \mathcal{C}^2_{2 \hookrightarrow 2} -
		\mathcal{C}^2_{\mu \hookrightarrow 2}\Gamma_{s,\tau}(\mathbf{f})
		-\mathfrak{O}(g,\mathbf{u}),
		\\
		\label{eq:apriori_3D_constants_3}
		\mathfrak{O}(g,\mathbf{u}) & := (\kappa_{-} \mathfrak{M})^{-1}\mathcal{L}_{\nu}\mathcal{C}_{6 \hookrightarrow 2}\mathcal{C}^2_{4\hookrightarrow 2}\|g\|_{H^{-1}(\Omega)}\| \nabla \mathbf{u} \|_{\mathbf{L}^3(\Omega)},
\end{align}
where $\mathcal{C}_{6 \hookrightarrow 2}$, $\mathcal{C}_{4 \hookrightarrow 2}$, $\mathcal{C}_{2 \hookrightarrow 2}$, and $\mathcal{C}_{\mu \hookrightarrow 2}$ are the best constants in $H^1_0(\Omega) \hookrightarrow L^{6}(\Omega)$, $\mathbf{H}^1_0(\Omega) \hookrightarrow \mathbf{L}^{4}(\Omega)$, $\mathbf{H}^1_0(\Omega) \hookrightarrow \mathbf{L}^{2}(\Omega)$, and $\mathbf{H}^1_0(\Omega) \hookrightarrow \mathbf{L}^{\mu}(\Omega)$, respectively. The constants $\nu_{-}$, $\kappa_{-}$, $\mathcal{L}_{\nu}$, and $\mathcal{L}_{\kappa}$ are defined in \S \ref{sec:main_assump}. $\mathcal{C}_N$ is the constant in \eqref{eq:estimates_navier}. $\Gamma_{s,\tau}$ is defined in \eqref{eq:Gamma_st} below. Finally, $\mu$ and $\tau$ are such that $2/\mu + 1/\tau = 1$, where $\tau = q$ for some $q > 1$ in two dimensions and $\tau = 3/2$ ($\mu = 6$) in three dimensions.

We are now ready to state and prove one of the main results of this section.

\begin{theorem}[error bound ($d=3$)]\label{thm:apriori_error_3D}
Let $d=3$, and let the assumptions of Theorems \ref{theorem_coupled_existence} and \ref{thm:uniq_coupled_3D_discrete} hold. If $\mathfrak{M}(T)>0$ and $\mathfrak{N}(\mathbf{f},g,\mathbf{u})>0$, then
we have the bound
\begin{multline}\label{eq:coupled_problem_best_approx}
\|\nabla \mathbf{e}_{\mathbf{u}}\|_{\mathbf{L}^2(\Omega)}
+
\|\nabla e_{T}\|_{\mathbf{L}^2(\Omega)}
+
\|e_{\mathsf{p}}\|_{L^2(\Omega)}\lesssim
\inf_{\mathbf{w}_h\in \mathbf{V}_h}\|\nabla(\mathbf{u}-\mathbf{w}_h)\|_{\mathbf{L}^2(\Omega)}
\\
+
\inf_{R_h\in Y_h} \|\nabla (T-R_h)\|_{\mathbf{L}^2(\Omega)}
+
\inf_{\mathsf{q}_h\in M_h}\|\mathsf{p}-\mathsf{q}_h\|_{L^2(\Omega)},
\end{multline}
with a hidden constant that is independent of $h$.
\end{theorem}

\begin{remark} \label{remark:Assumptions_smallness}[The assumptions $\mathfrak{M}(T)>0$ and $\mathfrak{N}(\mathbf{f},g,\mathbf{u})>0$ in three dimensions]
We note that the inequalities $\mathfrak{M}(T)>0$ and $\mathfrak{N}(\mathbf{f},g,\mathbf{u})>0$ can be rewritten as
\begin{equation*}
\frac{\mathcal{C}^2_{2 \hookrightarrow 2}}{\nu_{-}}+\frac{\mathcal{C}^2_{\mu \hookrightarrow 2}\Gamma_{s,\tau}(\mathbf{f})}{\nu_{-}}+
\frac{\mathcal{C}_N}{\nu_{-}^2}\|\mathbf{f}\|_{\mathbf{H}^{-1}(\Omega)} + \frac{\mathcal{L}_{\nu}\mathcal{C}_{6 \hookrightarrow 2}\mathcal{C}^2_{4\hookrightarrow 2}\|g\|_{H^{-1}(\Omega)}\| \nabla \mathbf{u} \|_{\mathbf{L}^3(\Omega)}}{\nu_{-}\kappa_{-}\left( \kappa_{-} - \mathcal{L}_{\kappa}\mathcal{C}_{6 \hookrightarrow 2} \| \nabla T \|_{{\mathbf{L}}^3(\Omega)} \right)}  < 1
\end{equation*}
and
$
\mathcal{L}_{\kappa}\mathcal{C}_{6 \hookrightarrow 2}\| \nabla T \|_{{\mathbf{L}}^3(\Omega)}  <  \kappa_{-},
$
respectively. This shows in particular that conditions  \eqref{eq:small_data_coupled} and \eqref{eq:small_data_coupled_2}, which guarantee the uniqueness of the coupled problem, are fulfilled.
\end{remark}

\begin{proof}

We divide the proof into four steps.

Step 1. \emph{A priori bound for the temperature error:} We control $\|\nabla e_{T}\|_{\mathbf{L}^2(\Omega)}$. For this purpose, we set $S = S_h\in Y_h \subset H_0^1(\Omega)$ in the temperature equation of the system \eqref{eq:modelweak} and subtract the third equation of the discrete system \eqref{eq:modelweak_discrete} from it to obtain
\begin{multline}\label{eq:error_temperature_coupled}
	\int_\Omega \left( \kappa(T)\nabla T - \kappa(T_h)\nabla T_h \right)\cdot\nabla S_h \mathrm{d}x+ \frac{1}{2}\int_\Omega \left( (\mathbf{u} \cdot \nabla T)  - (\mathbf{u}_h \cdot \nabla T_h) \right) S_h\mathrm{d}x\\
	+\frac{1}{2}\int_\Omega  \left((\mathbf{u}_h \cdot \nabla S_h)T_h  - (\mathbf{u} \cdot \nabla S_h)  T\right)\mathrm{d}x=0
	\quad
	\forall S_h \in Y_h,
\end{multline}
where we have used that $\mathbf{u} \in \mathbf{V}(\Omega)$ and the definition of $a_{ST}$ given in \eqref{eq:discrete_ST}. Now, let $R_h$ be an arbitrary element in $Y_h$. We rewrite the left-hand side of \eqref{eq:error_temperature_coupled} as follows:
\begin{multline}\label{eq:temp_reform}
	\int_\Omega \left( \kappa(T)\nabla T - \kappa(T_h)\nabla T_h \right)\cdot\nabla S_h \mathrm{d}x = \int_\Omega\left(\kappa(T)-\kappa(T_h)\right) \nabla T\cdot \nabla S_h \mathrm{d}x\\
	+\int_\Omega \kappa(T_h)\nabla ( T - R_h) \cdot\nabla S_h\mathrm{d}x+\int_\Omega \kappa(T_h)\nabla (R_h-T_h) \cdot\nabla S_h \mathrm{d}x
	\quad
	\forall S_h \in Y_h.
\end{multline}

In the next step, we set $S_h = R_h-T_h$ and replace \eqref{eq:temp_reform} in \eqref{eq:error_temperature_coupled} to obtain 
\begin{multline}\label{eq:error_temperature_coupled_reform}
	\int_\Omega \kappa(T_h)|\nabla ( R_h-T_h)|^2 \mathrm{d}x=-\int_\Omega\left(\kappa(T)-\kappa(T_h)\right) \nabla T\cdot \nabla (R_h-T_h)\mathrm{d}x
	\\-\int_\Omega \kappa(T_h)\nabla (T-R_h) \cdot\nabla (R_h-T_h)\mathrm{d}x  -\frac{1}{2}\int_\Omega \mathbf{u}\cdot \nabla (T-R_h) (R_h-T_h)\mathrm{d}x
	\\ -\frac{1}{2}\int_\Omega \mathbf{u} \cdot \nabla (R_h-T_h)(R_h-T_h)\mathrm{d}x
	-\frac{1}{2}\int_\Omega \left(\mathbf{u}-\mathbf{u}_h\right) \cdot \nabla T_h(R_h-T_h)\mathrm{d}x
	\\
	-
	\frac{1}{2}\int_\Omega \left( \mathbf{u}_h-\mathbf{u} \right)\cdot \nabla (R_h-T_h)T_h\mathrm{d}x
	+
	\frac{1}{2}\int_\Omega \mathbf{u} \cdot \nabla (R_h-T_h)(T -R_h)\mathrm{d}x
	\\
	+
	\frac{1}{2}\int_\Omega \mathbf{u} \cdot \nabla (R_h-T_h)(R_h-T_h)\mathrm{d}x.
\end{multline}
We now use that $\kappa_{-} \leq \kappa(r) \leq \kappa_{+}$ for every $r \in \mathbb{R}$, the Lipschitz property of $\kappa$, the Sobolev embeddings $H^1_0(\Omega) \hookrightarrow L^{6}(\Omega)$ and $\mathbf{H}^1_0(\Omega) \hookrightarrow \mathbf{L}^4(\Omega)$, the continuous and discrete stability bounds \eqref{eq:estimate_grad_u_coupled}--\eqref{eq:estimate_T_coupled} and \eqref{eq:estimate_grad_u_h}--\eqref{eq:estimate_T_h}, respectively, and the fact that $\mathbf{u} \in \mathbf{V}(\Omega)$
to obtain
\begin{multline*}\label{eq:first_estimate_heat}
	\kappa_{-}\|\nabla (T_h - R_h)\|_{\mathbf{L}^2(\Omega)}\leq \mathcal{L}_\kappa\mathcal{C}_{6 \hookrightarrow 2 } \|\nabla T \|_{\mathbf{L}^3(\Omega)}\|\nabla e_{T}\|_{\mathbf{L}^2(\Omega)}
	+
	\kappa_{+} \|\nabla (T - R_h)\|_{\mathbf{L}^2(\Omega)}
	\\ 
	+
	\mathcal{C}^2_{4 \hookrightarrow 2}\left(\kappa_{-}^{-1} \|g\|_{H^{-1}(\Omega)}\| \nabla \mathbf{e}_{\mathbf{u}}\|_{\mathbf{L}^2(\Omega)} + \nu_{-}^{-1}\|\mathbf{f}\|_{\mathbf{H}^{-1}(\Omega)}  \|\nabla (T - R_h)\|_{\mathbf{L}^2(\Omega)} \right).
\end{multline*}
The desired estimate for $\|\nabla e_{T}\|_{\mathbf{L}^2(\Omega)}$ follows from the triangle inequality and the assumption that $\mathfrak{M} = \mathfrak{M}(T) >0$. In fact, we have
\begin{multline}\label{eq:final_heat_error_estimate}
	\|\nabla e_{T}\|_{\mathbf{L}^2(\Omega)}
	\leq
	\mathfrak{M}^{-1}\left(\kappa_{-} + \kappa_{+} + \mathcal{C}^2_{4 \hookrightarrow 2}\nu_{-}^{-1}\|\mathbf{f}\|_{\mathbf{H}^{-1}(\Omega)} \right)\operatorname*{inf}_{R_h \in Y_h}\|\nabla(T-R_h)\|_{\mathbf{L}^2(\Omega)}\\
	+ (\mathfrak{M} \kappa_{-})^{-1}\mathcal{C}^2_{4\hookrightarrow 2} \|g\|_{H^{-1}(\Omega)}\|\nabla\mathbf{e}_{\mathbf{u}}\|_{\mathbf{L}^2(\Omega)}.
\end{multline}

\emph{Step 2. A priori bound for the pressure error:} We start with a simple application of the triangle inequality and write $\|e_\mathsf{p}\|_{L^2(\Omega)} \leq \|\mathsf{p} -\mathsf{q}_h\|_{L^2(\Omega)} + \|\mathsf{q}_h-\mathsf{p}_h\|_{L^2(\Omega)}$ for any $\mathsf{q}_h \in M_h$. It is therefore sufficient to control $\|\mathsf{q}_h-\mathsf{p}_h\|_{L^2(\Omega)}$. To do this, we first set $\mathbf{v}=\mathbf{v}_h\in \mathbf{X}_h$ in the first equation of problem \eqref{eq:modelweak} and subtract the first equation of problem \eqref{eq:modelweak_discrete} from it to obtain the following identity for any $\mathsf{q}_h \in M_h$:
\begin{multline}\label{eq:1st_presure_error_coupled}
	\int_{\Omega}\left(\mathsf{q}_h-\mathsf{p}_h\right) \operatorname{div} \mathbf{v}_h\mathrm{d}x = \int_{\Omega}\left(\mathsf{q}_h-\mathsf{p}\right) \operatorname{div} \mathbf{v}_h \mathrm{d}x +\int_{\Omega}(\mathbf{u}-\mathbf{u}_h)\cdot\mathbf{v}_h\mathrm{d}x\\
	+\int_\Omega (\nu(T)\nabla \mathbf{u}-\nu(T_h)\nabla \mathbf{u}_h)\cdot\nabla\mathbf{v}_h \mathrm{d}x + \int_{\Omega}(|\mathbf{u}|^{s-2}\mathbf{u}-|\mathbf{u}_h|^{s-2}\mathbf{u}_h)\cdot\mathbf{v}_h \mathrm{d}x
	\\ +\frac{1}{2}\int_\Omega((\mathbf{u} \cdot \nabla) \mathbf{u}-(\mathbf{u}_h \cdot \nabla) \mathbf{u}_h) \cdot\mathbf{v}_h \mathrm{d}x
	+\frac{1}{2}\int_\Omega \left( (\mathbf{u}_h \cdot \nabla)\mathbf{v}_h \cdot\mathbf{u}_h - (\mathbf{u} \cdot \nabla) \mathbf{v}_h\cdot\mathbf{u}\right) \mathrm{d}x
\end{multline}
for all $\mathbf{v}_h \in \mathbf{X}_h$. On the other hand, the inf-sup condition \eqref{eq:discrete_infsup} yields the existence of $\mathbf{w}_{h} \in \mathbf{X}_{h}$, so that $\operatorname{div} \mathbf{w}_h = \mathsf{q}_h - \mathsf{p}_h$ and $\|\nabla\mathbf{w}_h \|_{\mathbf{L}^2(\Omega)} \leq
(1/\tilde{\beta})\|\mathsf{q}_h-\mathsf{p}_h\|_{L^2(\Omega)}$. We therefore use \eqref{eq:1st_presure_error_coupled} with the particular function $\mathbf{w}_h$ as a test function, use the relations $\nu(T) \nabla \mathbf{u} - \nu(T_h) \nabla \mathbf{u}_h = (\nu(T) - \nu(T_h)) \nabla \mathbf{u} + \nu(T_h) \nabla (\mathbf{u}-\mathbf{u}_h)$ and $(\mathbf{u} \cdot \nabla) \mathbf{u} - (\mathbf{u}_h \cdot \nabla) \mathbf{u}_h = ((\mathbf{u} - \mathbf{u}_h) \cdot \nabla) \mathbf{u} + (\mathbf{u}_h \cdot \nabla) (\mathbf{u} - \mathbf{u}_h)$, and use basic estimates to obtain
\begin{multline}\label{eq:pressure_error_estimate}
	\|e_\mathsf{p}\|_{L^2(\Omega)}
	\leq
	\left[1 + \frac{1}{\tilde{\beta}}\right]\operatorname*{inf}_{\mathsf{q}_h \in M_h}
	\| \mathsf{p} - \mathsf{q}_h \|_{L^2(\Omega)}
	+
	\frac{1}{\tilde{\beta}}\mathcal{L}_\nu \mathcal{C}_{6 \hookrightarrow 2} \|\nabla e_{T}\|_{\mathbf{L}^2(\Omega)} \| \nabla \mathbf{u} \|_{\mathbf{L}^3(\Omega)}
	\\
	+ \frac{1}{\tilde{\beta}}
	\left(
	\mathcal{C}^2_{2 \hookrightarrow 2}
	+
	\nu_{+}
	+
	2\nu_{-}^{-1}\mathcal{C}_N\|\mathbf{f}\|_{\mathbf{H}^{-1}(\Omega)}
	+
	\mathcal{C}^2_{\mu \hookrightarrow 2 }\Gamma_{s,\tau}(\mathbf{f}) \right)\|\nabla \mathbf{e}_{\mathbf{u}}\|_{\mathbf{L}^2(\Omega)},
\end{multline}
where we have used \eqref{eq:estimates_a_weakly_cont} and \eqref{eq:Lambda_st} to control the Forchheimer term and obtain
\begin{equation}
\begin{aligned}
\label{eq:Gamma_st}
	\||\mathbf{u}_h| +  |\mathbf{u}|\|^{s-2}_{\mathbf{L}^{\tau(s-2)}(\Omega)}
	& \leq
    \mathcal{C}^{s-2}_{\tau(s-2) \hookrightarrow 2}
	\left( \| \nabla \mathbf{u}_h \|_{\mathbf{L}^{2}(\Omega)}
	+
	\| \nabla \mathbf{u} \|_{\mathbf{L}^{2}(\Omega)}
	\right)^{s-2}
	\\
	&
	\leq
	\mathcal{C}^{s-2}_{\tau(s-2) \hookrightarrow 2} 2^{s-2} \nu_{-}^{2-s} \|\mathbf{f}\|^{s-2}_{\mathbf{H}^{-1}(\Omega)}
	=:\Gamma_{s,\tau}(\mathbf{f}).
	\end{aligned}
\end{equation}
Here, $\mathcal{C}_{\tau(s-2) \hookrightarrow 2}$ is the best constant in $\mathbf{H}_0^1(\Omega) \hookrightarrow\mathbf{L}^{\tau(s-2)}(\Omega)$, $\mu = 6$, and $\tau = 3/2$.

\emph{Step 3. A priori bound for the velocity error:} We control $\|\nabla \mathbf{e}_{\mathbf{u}}\|_{\mathbf{L}^2(\Omega)}$. Our procedure is based on the arguments developed in step 1: We set $\mathbf{v} = \mathbf{v}_h\in \mathbf{X}_h$ in the first equation of \eqref{eq:modelweak} and subtract the first equation of \eqref{eq:modelweak_discrete} from it to obtain
\begin{multline}\label{eq:error_equations_coupled}
\int_\Omega (\nu(T)\nabla \mathbf{u}-\nu(T_h)\nabla \mathbf{u}_h)\cdot\nabla\mathbf{v}_h\mathrm{d}x
+
\frac{1}{2}
\int_\Omega[(\mathbf{u} \cdot \nabla) \mathbf{u}-(\mathbf{u}_h \cdot \nabla) \mathbf{u}_h] \cdot\mathbf{v}_h\mathrm{d}x\\
+
\frac{1}{2}
\int_\Omega [ (\mathbf{u}_h \cdot \nabla)\mathbf{v}_h \cdot\mathbf{u}_h - (\mathbf{u} \cdot \nabla) \mathbf{v}_h\cdot\mathbf{u}]
\mathrm{d}x
+
\int_{\Omega}\mathbf{e}_{\mathbf{u}}\cdot\mathbf{v}_h\mathrm{d}x\\
+\int_{\Omega}(|\mathbf{u}|^{s-2}\mathbf{u}-|\mathbf{u}_h|^{s-2}\mathbf{u}_h)\cdot\mathbf{v}_h\mathrm{d}x  -\int_\Omega
e_{\mathsf{p}}
\operatorname*{div} \mathbf{v}_h \mathrm{d}x =0 \quad \forall \mathbf{v}_h \in \mathbf{X}_h.
\end{multline}
Let $\mathbf{w}_h \in \mathbf{V}_h$. We rewrite the first term on the left-hand side of \eqref{eq:error_equations_coupled} as follows:
\begin{multline}\label{eq:nu_reform}
\int_\Omega (\nu(T)\nabla \mathbf{u}-\nu(T_h)\nabla \mathbf{u}_h)\cdot\nabla\mathbf{v}_h \mathrm{d}x = \int_\Omega \left(\nu(T)-\nu(T_h)\right) \nabla \mathbf{u} \cdot\nabla \mathbf{v}_h\mathrm{d}x
\\ 
+ \int_\Omega \nu(T_h)\nabla (\mathbf{u}-\mathbf{w}_h) \cdot\nabla \mathbf{v}_h\mathrm{d}x
		+ \int_\Omega \nu(T_h)\nabla(\mathbf{w}_h-\mathbf{u}_h) \cdot\nabla \mathbf{v}_h\mathrm{d}x.
\end{multline}
We now let $\mathsf{q}_h\in M_h$, set $\mathbf{v}_h = \mathbf{u}_h-\mathbf{w}_h\in \mathbf{V}_h$, and use \eqref{eq:nu_reform} to obtain
\begin{multline}\label{eq:error_equations_coupled_re}
\int_\Omega \nu(T_h)|\nabla(\mathbf{u}_h-\mathbf{w}_h)|^2\mathrm{d}x = \int_\Omega \left(\nu(T)-\nu(T_h)\right) \nabla \mathbf{u}\cdot \nabla (\mathbf{u}_h-\mathbf{w}_h)\mathrm{d}x
\\
+
\int_\Omega \nu(T_h)\nabla (\mathbf{u}-\mathbf{w}_h) \cdot\nabla (\mathbf{u}_h-\mathbf{w}_h)\mathrm{d}x
+
\frac{1}{2}\int_\Omega[(\mathbf{u} \cdot \nabla) \mathbf{u}-(\mathbf{u}_h \cdot \nabla) \mathbf{u}_h] \cdot(\mathbf{u}_h-\mathbf{w}_h)\mathrm{d}x
\\
+
\frac{1}{2}\int_\Omega [(\mathbf{u}_h \cdot \nabla)(\mathbf{u}_h-\mathbf{w}_h) \cdot\mathbf{u}_h - (\mathbf{u} \cdot \nabla) (\mathbf{u}_h-\mathbf{w}_h)\cdot\mathbf{u}]
\mathrm{d}x
+\int_{\Omega}\mathbf{e}_{\mathbf{u}}\cdot(\mathbf{u}_h-\mathbf{w}_h)\mathrm{d}x
\\
+\int_{\Omega}(|\mathbf{u}|^{s-2}\mathbf{u}-|\mathbf{u}_h|^{s-2}\mathbf{u}_h)\cdot(\mathbf{u}_h-\mathbf{w}_h)\mathrm{d}x - \int_\Omega \left( \mathsf{p}- \mathsf{q}_h \right) \operatorname*{div}(\mathbf{u}_h-\mathbf{w}_h)\mathrm{d}x.
\end{multline}
We now add $\pm \frac{1}{2}\int_{\Omega}(\mathbf{u}\cdot \nabla)\mathbf{u}_h\cdot(\mathbf{u}_h-\mathbf{w}_h)\mathrm{d}x$ and $\pm \frac{1}{2}\int_{\Omega}(\mathbf{u}\cdot \nabla)\mathbf{w}_h\cdot(\mathbf{u}_h-\mathbf{w}_h)\mathrm{d}x$ to rearrange the first difference of the convective terms as follows:
\begin{multline*}
	\frac{1}{2}\int_\Omega[(\mathbf{u} \cdot \nabla) \mathbf{u}-(\mathbf{u}_h \cdot \nabla) \mathbf{u}_h] \cdot(\mathbf{u}_h-\mathbf{w}_h)\mathrm{d}x = \frac{1}{2}\int_\Omega(\mathbf{u} \cdot \nabla) (\mathbf{u}-\mathbf{w}_h)\cdot(\mathbf{u}_h-\mathbf{w}_h)\mathrm{d}x\\+ \frac{1}{2}\int_\Omega(\mathbf{u} \cdot \nabla) (\mathbf{w}_h-\mathbf{u}_h) \cdot(\mathbf{u}_h-\mathbf{w}_h)\mathrm{d}x + \frac{1}{2}\int_\Omega((\mathbf{u}-\mathbf{u}_h )\cdot \nabla) \mathbf{u}_h \cdot(\mathbf{u}_h-\mathbf{w}_h)\mathrm{d}x.
\end{multline*}
A similar argument can be applied to the second difference of the convective terms by using the terms $\pm \frac{1}{2}\int_{\Omega}(\mathbf{u}\cdot \nabla)(\mathbf{u}_h-\mathbf{w}_h)\cdot \mathbf{u}_h\mathrm{d}x$ and $\pm \frac{1}{2}\int_{\Omega}(\mathbf{u}\cdot \nabla)(\mathbf{u}_h-\mathbf{w}_h)\cdot\mathbf{w}_h\mathrm{d}x$:
\begin{multline*}
	\frac{1}{2}\int_\Omega [(\mathbf{u}_h \cdot \nabla)(\mathbf{u}_h-\mathbf{w}_h) \cdot\mathbf{u}_h - (\mathbf{u} \cdot \nabla) (\mathbf{u}_h-\mathbf{w}_h)\cdot\mathbf{u}]
	\mathrm{d}x
	 =
	 \frac{1}{2}\int_\Omega((\mathbf{u}_h-\mathbf{u}) \cdot \nabla) (\mathbf{u}_h-\mathbf{w}_h)\cdot\mathbf{u}_h\mathrm{d}x
	 \\
	 + \frac{1}{2}\int_\Omega(\mathbf{u} \cdot \nabla) (\mathbf{u}_h-\mathbf{w}_h) \cdot(\mathbf{u}_h-\mathbf{w}_h)\mathrm{d}x
	 +
	 \frac{1}{2}\int_\Omega(\mathbf{u} \cdot \nabla) (\mathbf{u}_h-\mathbf{w}_h) \cdot(\mathbf{w}_h- \mathbf{u})\mathrm{d}x.
\end{multline*}
With these ingredients in hand, we invoke the assumptions of $\nu$ presented in \S \ref{sec:main_assump}, standard Sobolev embeddings, the bounds \eqref{eq:estimates_a_weakly_cont} and \eqref{eq:Gamma_st}, and the continuous and discrete stability bounds \eqref{eq:estimate_grad_u_coupled} and \eqref{eq:estimate_grad_u_h}, respectively, to conclude that
	\begin{multline}\label{eq:fluid_error_estimate}
		\nu_-\|\nabla ( \mathbf{u}_h-\mathbf{w}_h)\|_{\mathbf{L}^2(\Omega)} \leq \mathcal{L}_\nu \mathcal{C}_{6 \hookrightarrow 2} \|\nabla e_{T}\|_{\mathbf{L}^2(\Omega)} \| \nabla \mathbf{u} \|_{\mathbf{L}^3(\Omega)}+ \| \mathsf{p} - \mathsf{q}_h \|_{L^2(\Omega)}\\ 
		+ \left(\nu_+ + \nu_{-}^{-1} \mathcal{C}_N \|\mathbf{f}\|_{\mathbf{H}^{-1}(\Omega)}\right) \|\nabla ( \mathbf{u}-\mathbf{w}_h) \|_{\mathbf{L}^2(\Omega)} \\
		+ \left(\mathcal{C}^2_{2 \hookrightarrow 2}+\mathcal{C}^{2}_{\mu \hookrightarrow 2}\Gamma_{s,\tau}(\mathbf{f})+ \nu_{-}^{-1} \mathcal{C}_N \|\mathbf{f}\|_{\mathbf{H}^{-1}(\Omega)}\right)\|\nabla \mathbf{e}_{\mathbf{u}} \|_{\mathbf{L}^2(\Omega)}.
	\end{multline}
With this estimate at hand, we control $\|\nabla \mathbf{e}_{\mathbf{u}}\|_{\mathbf{L}^2(\Omega)}$ in view of the triangle inequality, the estimates \eqref{eq:final_heat_error_estimate} and \eqref{eq:fluid_error_estimate}, and the assumption $\mathfrak{N} = \mathfrak{N}(\mathbf{f},g,\mathbf{u})>0$. In fact, we have
\begin{multline}\label{eq:fluid_error_estimate_vel}
\|\nabla \mathbf{e}_{\mathbf{u}}\|_{\mathbf{L}^2(\Omega)} \leq  \frac{1}{\mathfrak{N}}\operatorname*{inf}_{\mathsf{q}_h \in M_h}  \| \mathsf{q}_h - \mathsf{p} \|_{L^2(\Omega)} \\
\frac{1}{\mathfrak{N}}\left(\frac{\kappa_{-}+\kappa_{+} +	\mathcal{C}^2_{4 \hookrightarrow 2}\nu_{-}^{-1}\|\mathbf{f}\|_{\mathbf{H}^{-1}(\Omega)}}{\mathfrak{M}}\right)\mathcal{L}_{\nu}\mathcal{C}_{6 \hookrightarrow 2} \|\nabla \mathbf{u}\|_{\mathbf{L}^3(\Omega)}\operatorname*{inf}_{R_h \in Y_h} \| \nabla (T - R_h) \|_{\mathbf{L}^2(\Omega)}
\\ + \left(\frac{\nu_{-} + \nu_{+} + \nu_{-}^{-1} \mathcal{C}_N \|\mathbf{f}\|_{\mathbf{H}^{-1}(\Omega)}}{\mathfrak{N}}\right)\operatorname*{inf}_{\mathbf{w}_h \in \mathbf{V}_h} \| \nabla (\mathbf{u} - \mathbf{w}_h) \|_{\mathbf{L}^2(\Omega)}.
\end{multline}

\emph{Step 4.} The desired estimate \eqref{eq:coupled_problem_best_approx} results from the combination of the bounds \eqref{eq:final_heat_error_estimate}, \eqref{eq:pressure_error_estimate} and \eqref{eq:fluid_error_estimate_vel}. This concludes the proof.
\end{proof}

Before presenting a bound in two dimensions, we introduce the quantities
\begin{equation}
\begin{aligned}
\mathfrak{P}(T)
	& := \kappa_{-} - \mathcal{L}_{\kappa}\mathcal{C}_{\varepsilon}\| \nabla T \|_{{\mathbf{L}}^{2+\varepsilon}(\Omega)} ,
\label{eq:small_data_coupled_error_2D}
\\
\mathfrak{Q}(\mathbf{f},g,\mathbf{u})
& :=
\nu_{-}
-
\nu_{-}^{-1}\mathcal{C}_N\|\mathbf{f}\|_{\mathbf{H}^{-1}(\Omega)}
-
\mathcal{C}^2_{2 \hookrightarrow 2}
-
\mathcal{C}^2_{\mu \hookrightarrow 2}\Gamma_{s,\tau}(\mathbf{f})
-
\mathfrak{R}(g,\mathbf{u}),
\\
\mathfrak{R}(g,\mathbf{u}) & := (\kappa_{-} \mathfrak{M})^{-1}\mathcal{L}_{\nu}\mathcal{C}_{\varepsilon}\mathcal{C}^2_{4\hookrightarrow 2}\|g\|_{H^{-1}(\Omega)}\| \nabla \mathbf{u} \|_{\mathbf{L}^{2+\epsilon}(\Omega)},
\end{aligned}
\end{equation}
for some $\varepsilon>0$. $\mathcal{C}_{\varepsilon}$ corresponds to the best constant in $H^1_0(\Omega) \hookrightarrow L^{2(2+\varepsilon)/\varepsilon}(\Omega)$.
\begin{theorem}[error bound $(d=2)$]
\label{thm:apriori_error_2D}
Let $d=2$, and let the assumptions of Theorems \ref{theorem_coupled_existence} and \ref{thm:uniq_coupled_2D_discrete} hold. If $\mathfrak{P}(T)>0$ and $\mathfrak{Q}(\mathbf{f},g,\mathbf{u})>0$, then we have the bound
\begin{multline}\label{eq:coupled_problem_best_approx_2D}
\|\nabla \mathbf{e}_{\mathbf{u}}\|_{\mathbf{L}^2(\Omega)}+\|\nabla e_{T}\|_{\mathbf{L}^2(\Omega)}+ \|e_{\mathsf{p}}\|_{L^2(\Omega)}\lesssim
\inf_{\mathbf{w}_h\in \mathbf{V}_h}\|\nabla(\mathbf{u}-\mathbf{w}_h)\|_{\mathbf{L}^2(\Omega)}
\\
+
\inf_{R_h\in Y_h} \|\nabla (T-R_h)\|_{\mathbf{L}^2(\Omega)}
+
\inf_{\mathsf{q}_h\in M_h}\|\mathsf{p}-\mathsf{q}_h\|_{L^2(\Omega)},
\end{multline}
with a hidden constant that is independent of $h$.
\end{theorem}
\begin{proof}
The proof follows from a slight modification of the arguments developed in the proof of Theorem \ref{thm:apriori_error_3D}, which essentially consists in using the Sobolev embedding $H^1_0(\Omega) \hookrightarrow L^{2(2+\varepsilon)/\varepsilon}(\Omega)$ instead of $H^1_0(\Omega) \hookrightarrow L^{6}(\Omega)$. For the sake of brevity, we omit the details.
\end{proof}

\subsection{A priori error bounds}
As a direct consequence of the best approximation results of Theorems \ref{thm:apriori_error_3D} and \ref{thm:apriori_error_2D}, approximation theory yields the following convergence rates for Taylor--Hood and mini element approximation.

\begin{theorem}[Taylor--Hood approximation]\label{theorem_apriori1}
Let the assumptions of Theorem \ref{thm:apriori_error_3D} and \ref{thm:apriori_error_2D} hold in three and two dimensions, respectively. If the solution $(\mathbf{u},\mathsf{p},T)$ to \eqref{eq:modelweak} belongs to $\mathbf{H}^{3}(\Omega)\cap \mathbf{H}_0^1(\Omega) \times H^2(\Omega)\cap L_0^2(\Omega) \times H^{3}(\Omega)\cap H_0^1(\Omega)$, then
\begin{align*}
\| \nabla \mathbf{e}_{\mathbf{u}} \|_{\mathbf{L}^2(\Omega)} + \| e_{\mathsf{p}}\|_{L^2(\Omega)}
+
\| \nabla e_{T}\|_{\mathbf{L}^2(\Omega)}
\lesssim h^2
\left( \|\mathbf{u}\|_{\mathbf{H}^{3}(\Omega)}+ \|\mathsf{p}\|_{H^{2}(\Omega)}+\|T\|_{H^{3}(\Omega)}  \right).
\end{align*}
\end{theorem}

\begin{theorem}[mini element approximation]
Let the assumptions of Theorem \ref{thm:apriori_error_3D} and \ref{thm:apriori_error_2D} hold in three and two dimensions, respectively. If the solution $(\mathbf{u},\mathsf{p},T)$ to \eqref{eq:modelweak} belongs to $\mathbf{H}^{2}(\Omega)\cap \mathbf{H}_0^1(\Omega) \times H^1(\Omega)\cap L_0^2(\Omega) \times H^{2}(\Omega)\cap H_0^1(\Omega)$, then
\begin{align*}
\| \nabla \mathbf{e}_{\mathbf{u}} \|_{\mathbf{L}^2(\Omega)} + \| e_{\mathsf{p}}\|_{L^2(\Omega)}+\| \nabla e_{T}\|_{\mathbf{L}^2(\Omega)} \lesssim h \left( \|\mathbf{u}\|_{\mathbf{H}^{2}(\Omega)}+ \|\mathsf{p}\|_{H^{1}(\Omega)}+\|T\|_{H^{2}(\Omega)}  \right).
\end{align*}
\end{theorem}


\section{Numerical examples}
\label{sec:numericalexample}

In this section, we present several numerical experiments that illustrate the performance of the developed finite element method. These examples were performed with a code that we implemented in \texttt{C++}. All matrices were assembled exactly, and global linear systems were solved with the multifrontal massively parallel sparse direct solver (MUMPS) \cite{MUMPS1,MUMPS2}. For the visualization of finite element approximations, we used the open-source application ParaView \cite{Ayachit2015ThePG}.

Given a mesh $\T_h$, we approximate the velocity and pressure of the fluid with the Taylor--Hood element \eqref{TH:vel_space}--\eqref{TH:press_space} and the temperature of the fluid with functions in the space described in \eqref{eq:discrete_temp_space} with $\mathfrak{r} = 2$. We solve the nonlinear system \eqref{eq:modelweak_discrete} using the fixed-point strategy described in \textbf{Algorithm 1}. To investigate experimental convergence rates, we define the total number of degrees of freedom as follows: $\text{Ndof}:=\text{dim}(\mathbf{X}_h) + \text{dim}(M_h) + \text{dim}(Y_h).$

\begin{algorithm}[ht]
\caption{\textbf{Fixed-point iteration}}
\label{Algorithm1}
\textbf{Input:} Mesh $\mathscr{T}$, initial guess $(\mathbf{u}_{h}^{0},\mathsf{p}_{h}^{0},T_{h}^{0}) \in \mathbf{X}_h \times M_h \times Y_h$, $\mathbf{f}\in\mathbf{H}^{-1}(\Omega)$, $g\in H^{-1}(\Omega)$, $s \in [3,4]$, and $\textrm{tol} = 10^{-6}$;
\\
$\boldsymbol{1}$: For $i\geq 0$, find $(\mathbf{u}_h^{i+1},\mathsf{p}_h^{i+1},T_h^{i+1}) \in \mathbf{X}_h \times M_h \times Y_h$ such that
\begin{multline*}
\mathfrak{a}(T_h^{i};\mathbf{u}_h^{i+1}, \mathbf{v}_h)
+
a_{SN}(\mathbf{u}_h^{i};\mathbf{u}_h^{i+1},\mathbf{v}_h)
+
a_{F}(\mathbf{u}_h^{i};\mathbf{u}_h^{i+1},\mathbf{v}_h)
+
b(\mathbf{v}_h,\mathsf{p}_h^{i+1})
=\langle \mathbf{f}, \mathbf{v}_h \rangle,
\\
b(\mathbf{u}_h^{i+1},\mathsf{q}_h)=0
\qquad \forall (\mathbf{v}_h,\mathsf{q}_h) \in \mathbf{X}_h \times M_h.
\end{multline*}
Then, $T_h^{i+1}\in Y_h$ is found as the solution of the problem
\begin{align*}
\mathfrak{b}(T_h^{i};T_h^{i+1}, S_h)+a_{ST}(\mathbf{u}_h^{i+1};T_h^{i+1},S_h)=\langle g,S_h\rangle
\qquad
\forall S_h \in Y_h.
\end{align*}

$\boldsymbol{2}$: If $|(\mathbf{u}_{h}^{i+1}, \mathsf{p}_{h}^{i+1}, T_{h}^{i+1})-(\mathbf{u}_{h}^{i}, \mathsf{p}_{h}^{i}, T_{h}^{i} )| > \textrm{tol}$,  set $i \leftarrow i + 1$ and go to step $\boldsymbol{1}$. Otherwise, return $(\mathbf{u}_{h}, \mathsf{p}_{h}, T_{h}) = (\mathbf{u}_{h}^{i+1}, \mathsf{p}_{h}^{i+1}, T_{h}^{i+1} )$. Here, $|\cdot|$ denotes the Euclidean norm.
\end{algorithm}

Let us now show the convergence of this algorithm. For this purpose, we define the errors $\mathbf{e}^{i+1}_{\mathbf{u}_h}:=\mathbf{u}_h-\mathbf{u}^{i+1}_h$, $e^{i+1}_{\mathsf{p}_h}:=\mathsf{p}_h-\mathsf{p}^{i+1}_h$, $e^{i+1}_{T_h}:=T_h-T^{i+1}_h$, and
		\begin{align}
			\label{eq:algorithm_3D_constants_2}
			\mathfrak{V}(\mathbf{f},g)
			&:=
			\left(1 + \frac{\mathcal{C}^2_{4 \hookrightarrow2}\| g \|_{H^{-1}(\Omega)}}{\kappa_{-}^2}\right)\left(\frac{\nu_{-}^{-1}\mathcal{C}_N \|\mathbf{f}\|_{\mathbf{H}^{-1}(\Omega)}+2\mathcal{C}^2_{\mu \hookrightarrow 2}\Gamma_{s,\tau}(\mathbf{f})}{\nu_{-}-\mathcal{C}^2_{\mu \hookrightarrow 2}\Gamma_{s,\tau}(\mathbf{f})-\mathcal{C}^{2}_{2 \hookrightarrow2}}\right),
			\\\label{eq:algorith_3D_constants_1}
			\mathfrak{T}(\mathbf{f},g,\mathbf{u}_h,&T_h)
			:= \frac{\mathcal{L}_{\kappa}\mathcal{C}_{6 \hookrightarrow 2}\| \nabla T_h \|_{{\mathbf{L}}^3(\Omega)}}{\kappa_{-}} + \frac{\mathcal{L}_{\nu}\mathcal{C}_{6 \hookrightarrow 2}\mathcal{C}^2_{4\hookrightarrow 2}\|g\|_{H^{-1}(\Omega)}\| \nabla \mathbf{u}_h \|_{\mathbf{L}^3(\Omega)}}{\kappa_{-}^2\left(\nu_{-}-\mathcal{C}^2_{\mu \hookrightarrow 2}\Gamma_{s,\tau}(\mathbf{f})-\mathcal{C}^{2}_{2 \hookrightarrow2}\right)},
			\\
			\label{eq:algorithm_3D_constants_3}
			\mathfrak{A}(\mathbf{f},g,\mathbf{u}_h,&T_h)
			:=	\mathfrak{T}(\mathbf{f},g,\mathbf{u}_h,T_h) + \frac{\mathcal{L}_{\nu}\mathcal{C}_{6 \hookrightarrow 2}\| \nabla \mathbf{u}_h \|_{{\mathbf{L}}^3(\Omega)}}{\nu_{-}-\mathcal{C}^2_{\mu \hookrightarrow 2}\Gamma_{s,\tau}(\mathbf{f})-\mathcal{C}^{2}_{2 \hookrightarrow2}}.
	\end{align}
	
	We are now ready to present and prove the convergence of our algorithm.
	
	\begin{theorem}[convergence of the \textbf{Algorithm} 1]
		Let $d=3$, and let the assumptions of Theorems \ref{th:existence_coupled_problem_discrete}, \ref{thm:uniq_coupled_3D_discrete} and \ref{thm:apriori_error_3D} hold. If  $\Theta:=\max\left\{\mathfrak{V(\mathbf{f},g)},\mathfrak{A}(\mathbf{f},g,\mathbf{u}_h,T_h)\right\}<1$, then
		\begin{equation*}
			\lim _{i \rightarrow \infty}\|\nabla e^{i}_{T_h}\|_{\mathbf{L}^2(\Omega)}=0, 
			\qquad 
			\lim _{i \rightarrow \infty}\|\nabla \mathbf{e}^{i}_{\mathbf{u}_h}\|_{\mathbf{L}^2(\Omega)}=0, 
			\qquad
			\lim _{i \rightarrow \infty}\| e^{i}_{\mathsf{p}_h}\|_{L^2(\Omega)}=0.
		\end{equation*}
	\end{theorem}
	\begin{proof}
As a first step, we subtract from the discrete problem \eqref{eq:modelweak_discrete} the discrete equations that occur in \textbf{Algorithm 1}. We set $\mathbf{v}_h = \mathbf{u}^{i+1}_h-\mathbf{u}_h$ and $S_h=T^{i+1}_h-T_h$ into the resulting system and use standard estimates to obtain
\begin{multline}
\label{eq:first_alg_bound}
\mathfrak{W} \|\nabla \mathbf{e}^{i+1}_{\mathbf{u}_h}\|_{\mathbf{L}^2(\Omega)} \leq \mathcal{L}_\nu \mathcal{C}_{6 \hookrightarrow 2} \| \nabla \mathbf{u}_h \|_{\mathbf{L}^{3}(\Omega)} \|\nabla e^{i}_{T_h}\|_{\mathbf{L}^2(\Omega)}\\
+
\left(\nu_{-}^{-1}\mathcal{C}_N \|\mathbf{f}\|_{\mathbf{H}^{-1}(\Omega)}+2\mathcal{C}^2_{\mu \hookrightarrow 2}\Gamma_{s,\tau}(\mathbf{f})\right)\|\nabla \mathbf{e}^{i}_{\mathbf{u}_h}\|_{\mathbf{L}^2(\Omega)},
\end{multline}
and
\begin{multline}
\label{eq:second_alg_bound_aux}
\kappa_{-}\|\nabla e^{i+1}_{T_h}\|_{\mathbf{L}^2(\Omega)} \leq  \mathcal{L}_\kappa \mathcal{C}_{6 \hookrightarrow 2} \| \nabla T_h \|_{\mathbf{L}^{3}(\Omega)}\|\nabla e^{i}_{T_h}\|_{\mathbf{L}^2(\Omega)} \\
+ 
\mathcal{C}^2_{4\hookrightarrow 2}\kappa_{-}^{-1}\|g\|_{H^{-1}(\Omega)}\|\nabla \mathbf{e}^{i+1}_{\mathbf{u}_h}\|_{\mathbf{L}^2(\Omega)}.
\end{multline}
Here, $\mathfrak{W}:=\nu_{-}-\mathcal{C}^2_{\mu \hookrightarrow 2}\Gamma_{s,\tau}(\mathbf{f})-\mathcal{C}^{2}_{2 \hookrightarrow2}$. Note that $\mathfrak{W}>0$; see Remark \ref{remark:Assumptions_smallness}. If we replace the estimate \eqref{eq:first_alg_bound} in \eqref{eq:second_alg_bound_aux}, we obtain
\begin{multline}\label{eq:second_alg_bound}
\|\nabla e^{i+1}_{T_h}\|_{\mathbf{L}^2(\Omega)} \leq \mathfrak{T}(\mathbf{f},g,\mathbf{u}_h,T_h)\|\nabla e^{i}_{T_h}\|_{\mathbf{L}^2(\Omega)} \\
+ 
\frac{\mathcal{C}^2_{4 \hookrightarrow2}\| g \|_{H^{-1}(\Omega)}}{\kappa_{-}^2\mathfrak{W}}\left(
\nu_{-}^{-1}\mathcal{C}_N \|\mathbf{f}\|_{\mathbf{H}^{-1}(\Omega)}
+
2\mathcal{C}^2_{\mu \hookrightarrow 2}\Gamma_{s,\tau}(\mathbf{f})\right)\|\nabla \mathbf{e}^{i}_{\mathbf{u}_h}\|_{\mathbf{L}^2(\Omega)}.
\end{multline}
If we add the estimates \eqref{eq:first_alg_bound} and \eqref{eq:second_alg_bound}, we finally arrive at
\begin{multline*}\label{eq:fluid_error_estimate_algorithm}
\|\nabla\mathbf{e}^{i+1}_{\mathbf{u}_h}\|_{\mathbf{L}^2(\Omega)} + \|\nabla e^{i+1}_{T_h}\|_{\mathbf{L}^2(\Omega)} \leq \mathfrak{A}(\mathbf{f},g,\mathbf{u}_h,T_h)\|\nabla e^{i}_{T_h}\|_{\mathbf{L}^2(\Omega)} + \mathfrak{V}(\mathbf{f},g)\|\nabla\mathbf{e}^{i}_{\mathbf{u}_h}\|_{\mathbf{L}^2(\Omega)}.
\end{multline*}

Let us now introduce $\{z_{i}\}_{i=0}^{\infty}$, where
$
z_{i}:=\|\nabla\mathbf{e}^{i}_{\mathbf{u}_h}\|_{\mathbf{L}^2(\Omega)} + \|\nabla e^{i}_{T_h}\|_{\mathbf{L}^2(\Omega)}
$
for $i \geq 0$. Since $\Theta:=\max\left\{\mathfrak{V(\mathbf{f},g)},\mathfrak{A}(\mathbf{f},g,\mathbf{u}_h,T_h)\right\}<1$, we can conclude that $z_{i}\to 0$ as $i \rightarrow \infty$ because $z_i\leq \Theta^i z_0$ for $i \geq 1$. These arguments show that
\[
 \|\nabla\mathbf{e}^{i}_{\mathbf{u}_h}\|_{\mathbf{L}^2(\Omega)} \rightarrow 0,
 \qquad
 \|\nabla e^{i}_{T_h}\|_{\mathbf{L}^2(\Omega)} \rightarrow 0,
 \qquad i \uparrow \infty.
\]
The convergence result $\| e^{i}_{\mathsf{p}_h}\|_{L^2(\Omega)} \rightarrow 0$ as $i \uparrow \infty$ follows from the inf-sup condition. This completes the proof.
	\end{proof}

\subsection{Experimental convergence rates}
\label{subsec:error_convergence}
Let $\Omega=(0,1)^2$, $\kappa(T)=4+\sin(T)$, $\nu(T)=1+e^{-T^2}$, and $s\in\{3,3.5,4\}$. We note that $\nu(\cdot)$ and $\kappa(\cdot)$ fulfill the assumptions from Section \ref{sec:main_assump}. The data $\mathbf{f}$ and $g$ are chosen so that the exact solution of \eqref{eq:modelweak} is
\begin{multline*}
\mathbf{u}(x_1,x_2) = (-x_1^2(x_1 -1)^2x_2(x_2-1)(2x_2 -1),x_2^2(x_2 -1)^2x_1(x_1-1)(2x_1 -1)),
\\ 
\mathsf{p}(x_1,x_2) = x_1 x_2 (1-x_1)(1-x_2)-1/36,
\quad
T(x_1,x_2) = x_1^2 x_2^2 (1-x_1)^2(1-x_2)^2.
\end{multline*}

In Figure \ref{fig:test_01}, we show the experimental convergence rates for the errors that occur when approximating the velocity field, pressure, and temperature variables in appropriate norms. We observe \emph{optimal} experimental convergence rates for all errors and all values of the considered paremeter $s$ and confirm computationally the theoretical results obtained in Theorem \ref{theorem_apriori1}.

\begin{figure}[!ht]
\centering
\includegraphics[trim={0 0cm 0 0cm},clip,width=13.0cm,height=4.5cm,scale=0.66]{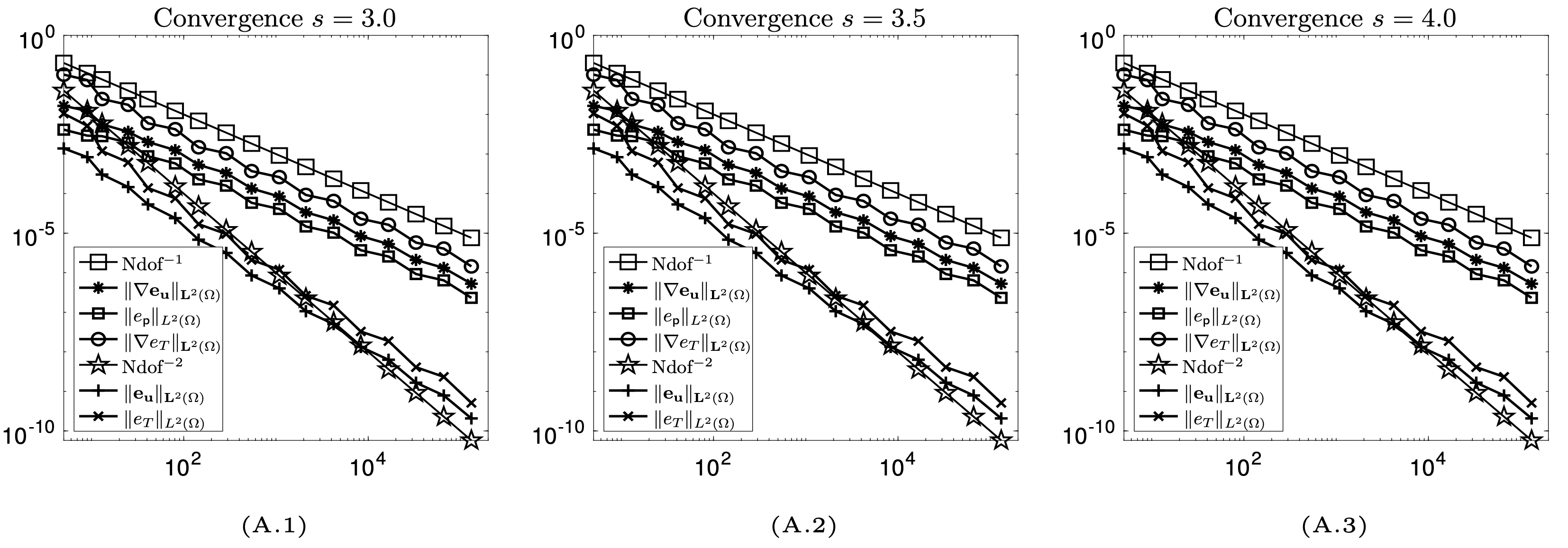} \\
\caption{Example 1. Performance of the developed finite element discretization scheme: Taylor--Hood approximation for the velocity and pressure variables and the space of continuous piecewise quadratic functions to approximate the temperature variable. We present experimental convergence rates for the errors that occur
when approximating the velocity field, pressure, and temperature variables in appropriate norms for $s=3.0$ (A.1), $s=3.5$ (A.2), and $s=4.0$ (A.3).}
\label{fig:test_01}
\end{figure}

\subsection{Heated lid-driven cavity flow problem}

The lid-driven cavity flow problem is probably one of the most studied problems in the field of computational fluid dynamics. The simplicity of the geometry of the cavity flow simplifies the numerical implementation and also the consideration of different boundary conditions. Even though the problem looks simple in many respects, the flow in a cavity retains all the flow physics, with counter-rotating vortices occurring inside of the cavity.

We consider the domain $\Omega=(0,1)^2$, the diffusion coefficient $\kappa(T)=4+\sin(T)$, the viscosity coefficient $\nu(T)=1/400 + \exp(-T^2)$, and the parameter $s\in\{3.0,3.5,4.0\}$. The forcing terms are $\mathbf{f}=(0,0)^{\intercal}$ and $g=0$. In this numerical example, we omit the reaction term $\mathbf{u}$ in \eqref{eq:modelweak} and investigate the behavior of the solutions of the coupled problem beyond the theory that we have presented since we use the homogeneous Neumann and nonhomogeneous Dirichlet boundary conditions described in Figure \ref{fig:domain_square}.

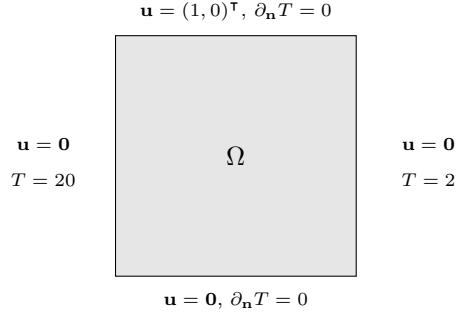
\begin{figure}[h!]
\centering
\begin{tikzpicture}[scale=3.2]
    \draw[black, fill=black!10] (0, 0) rectangle (1, 1);
    
    \node at (1.3, 0.55) {\scriptsize $\mathbf{u}=\mathbf{0}$};
    \node at (1.3, 0.4) {\scriptsize $T=2$};

    \node at (-0.3, 0.55) {\scriptsize $\mathbf{u}=\mathbf{0}$};
    \node at (-0.3, 0.4) {\scriptsize $T=20$};

    \node at (0.5, -0.1) {\scriptsize $\mathbf{u}=\mathbf{0}$, $\partial_{\mathbf{n}} T=0$};

    \node at (0.5, 1.1) {\scriptsize $\mathbf{u}=(1,0)^{\intercal}$, $\partial_{\mathbf{n}} T=0$};
    \node at (0.5, 0.5) {$\Omega$};
\end{tikzpicture}
\caption{Boundary conditions for the cavity flow problem in the domain $\Omega=(0,1)^2$.}
\label{fig:domain_square}
\end{figure}	

In Figure \ref{fig:test_03}, we show the streamlines for the velocity field, the contour lines of the pressure, and the temperature variable. To illustrate the results, we consider a mesh with $65.536$ elements, which corresponds to $33.025$ nodes. We note that for each value of $s$, a circulation pattern is established within the cavity, driven by the temperature gradient between the right and left components of the boundary. This temperature distribution leads to convective flow, which is characteristic of cavity problems where temperature fluctuations along the side walls contribute to fluid motion. In this context, we show how the parameter $s$ affects the velocity and pressure of the fluid and, in particular,
the position of the counter-rotating vortices. In particular, the center of the main recirculation is located at $(0.65234,0.75782)$ for $s=3.0$, $(0.65234,0.75781)$ for $s=3.5$, and $(0.58398,0.68164)$ for $s=4.0$. Note also that the same effect occurs for the secondary recirculation, which is located in the lower right part of the cavity. Finally, we mention that in all cases, the pressure exhibits a singular behavior in the upper right and left corners of the cavity.

\begin{figure}[!ht]
\centering
\begin{minipage}[b]{0.327\textwidth}\centering
\includegraphics[trim={0.5cm 3cm 0 0cm},clip,width=4.4cm,height=3.95cm,scale=0.66]{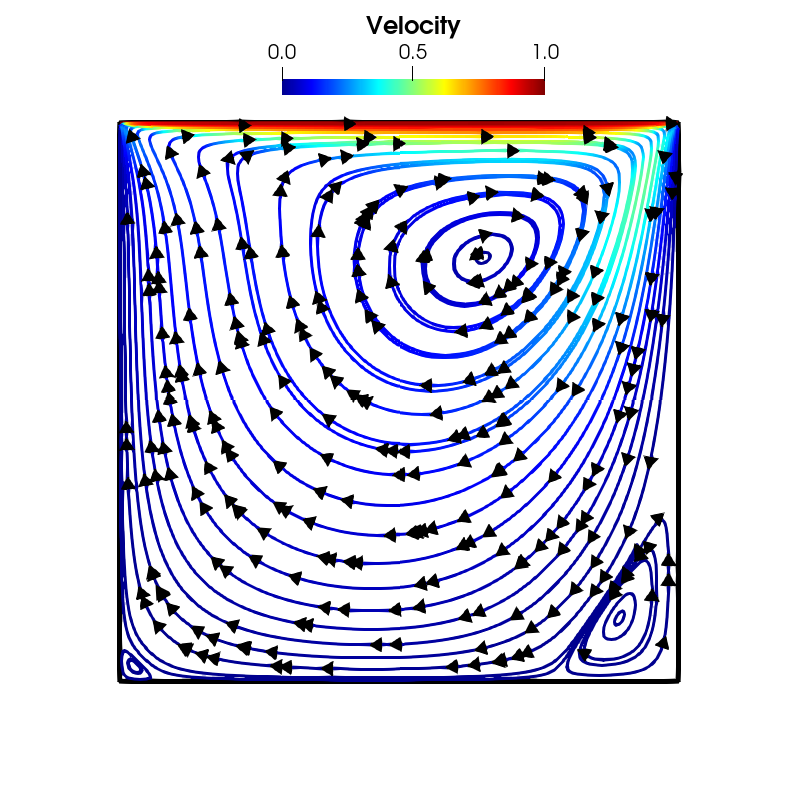} \\
\quad \tiny{(C.1)}
\end{minipage}
\begin{minipage}[b]{0.327\textwidth}\centering
\includegraphics[trim={0.5cm 3cm 0 0cm},clip,width=4.4cm,height=3.95cm,scale=0.66]{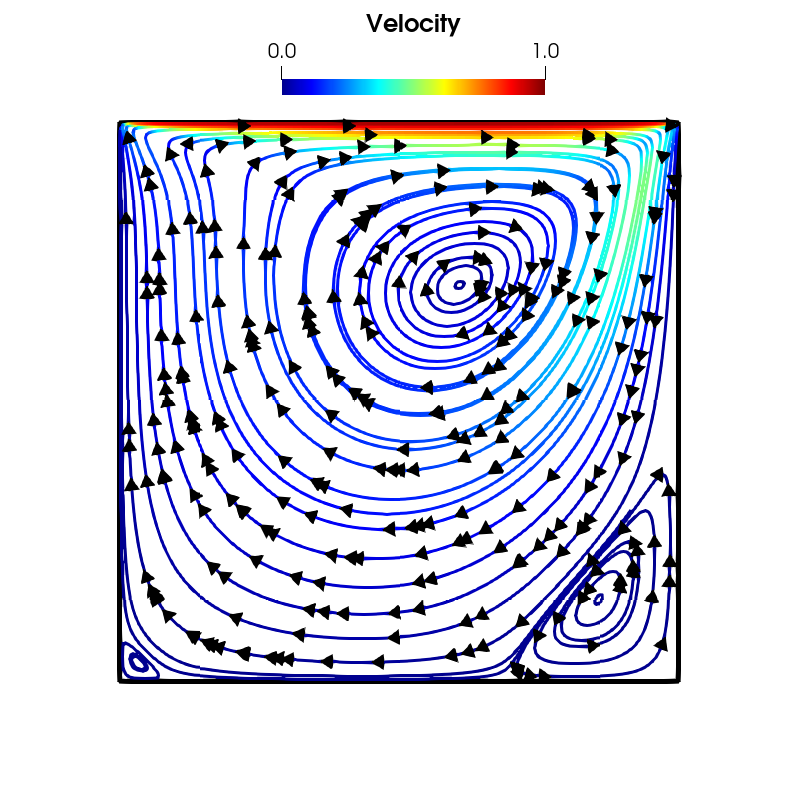} \\
\quad \tiny{(C.2)}
\end{minipage}
\begin{minipage}[b]{0.327\textwidth}\centering
\includegraphics[trim={0.5cm 3cm 0 0cm},clip,width=4.4cm,height=3.95cm,scale=0.66]{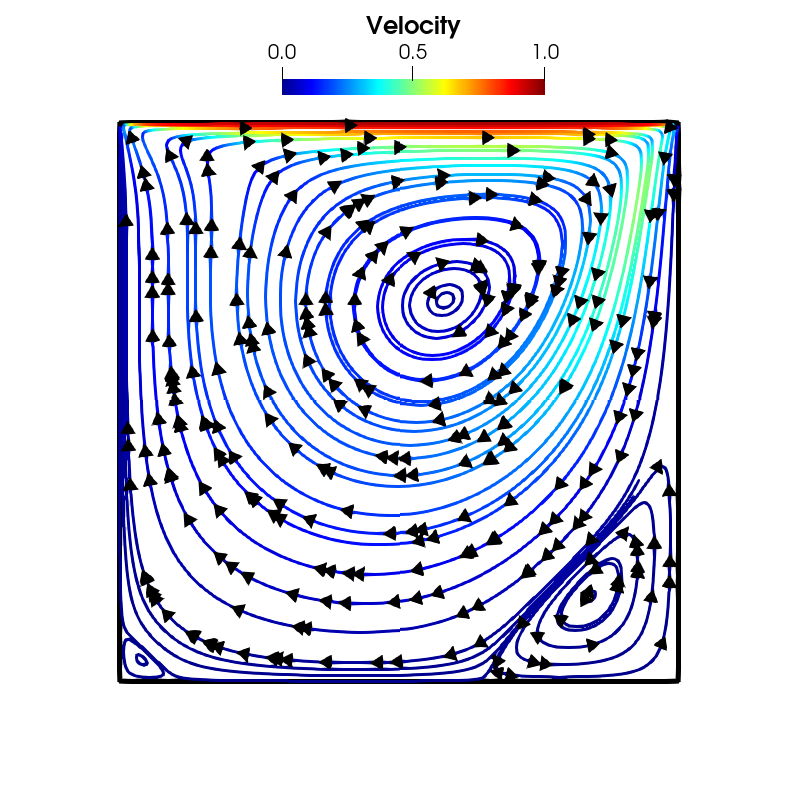} \\
\quad \tiny{(C.3)}
\end{minipage}
\\~\\
\begin{minipage}[b]{0.327\textwidth}\centering
\includegraphics[trim={4cm 3.0cm 4.0cm 0cm},clip,width=3.25cm,height=4.05cm,scale=0.66]{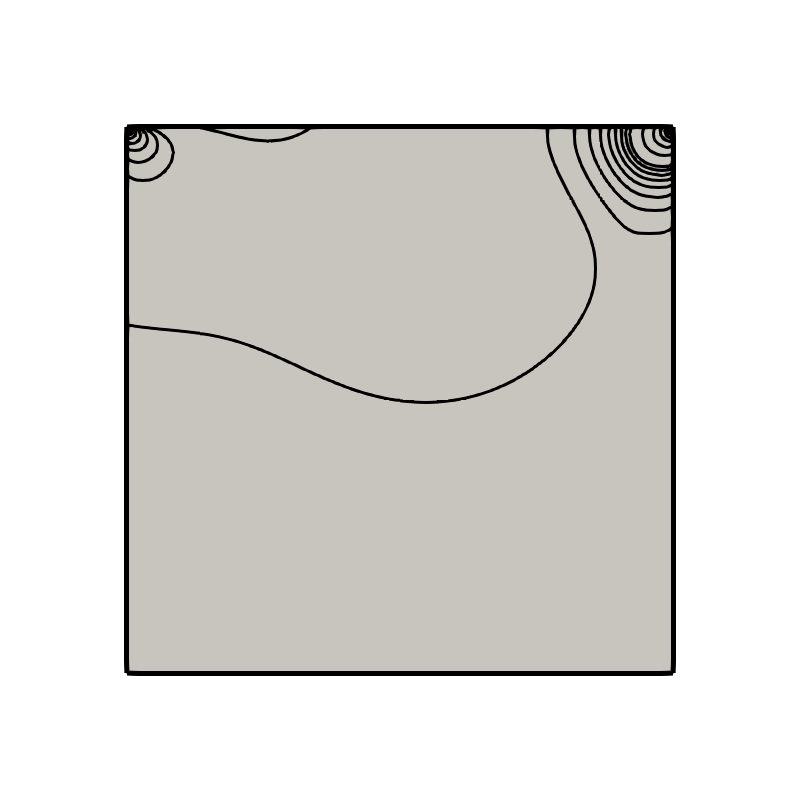} \\
\quad \tiny{(C.4)}
\end{minipage}
\begin{minipage}[b]{0.327\textwidth}\centering 
\includegraphics[trim={4cm 3.0cm 4.0cm 0cm},clip,width=3.25cm,height=4.05cm,scale=0.66]{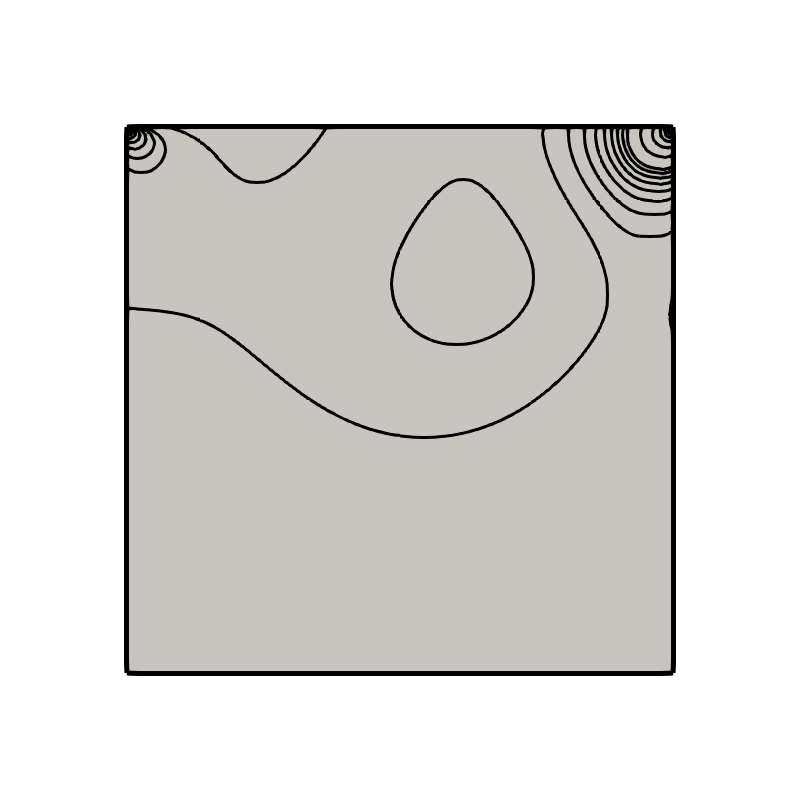} \\
\quad \tiny{(C.5)}
\end{minipage}
\begin{minipage}[b]{0.327\textwidth}\centering  
\includegraphics[trim={4cm 3.0cm 4.0cm 0cm},clip,width=3.25cm,height=4.05cm,scale=0.66]{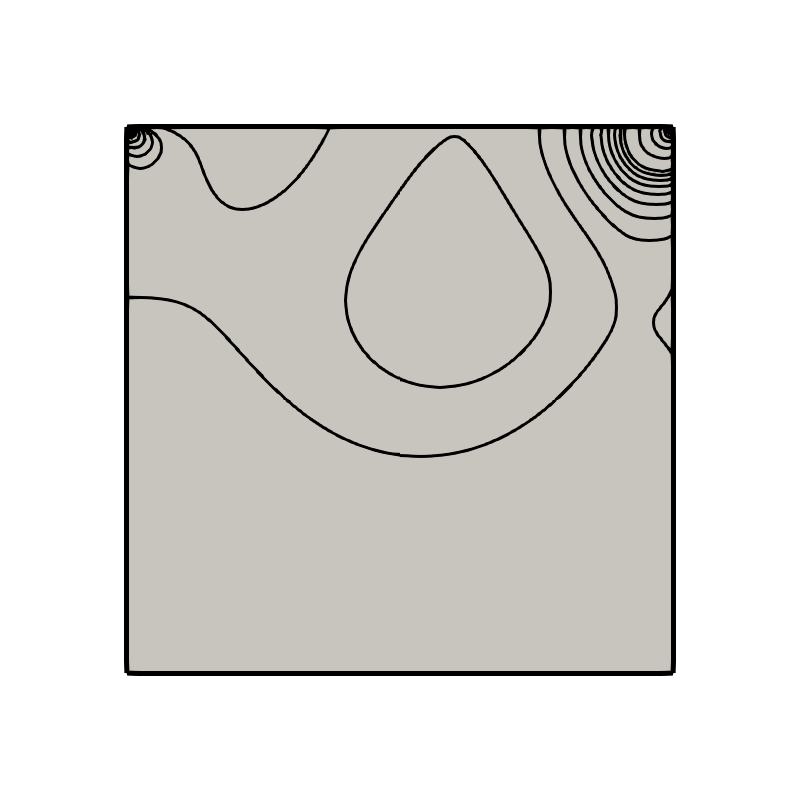} \\
\quad \tiny{(C.6)}
\end{minipage}
\\~\\
\begin{minipage}[b]{0.327\textwidth}\centering
\includegraphics[trim={4cm 3.0cm 4.0cm 0cm},clip,width=3.25cm,height=4.05cm,scale=0.66]{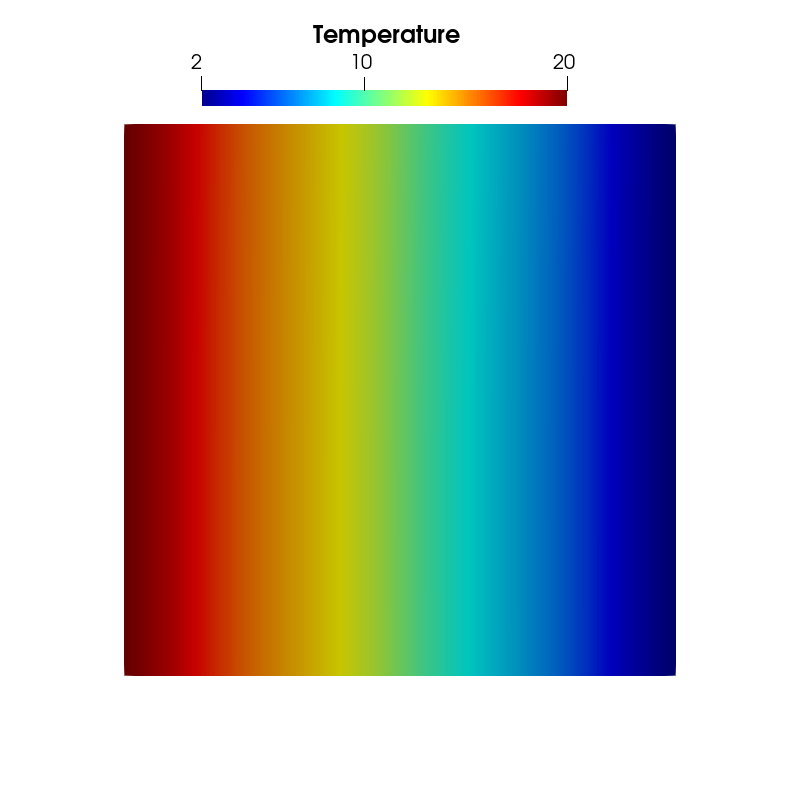} \\
\quad \tiny{(C.7)}
\end{minipage}
\begin{minipage}[b]{0.327\textwidth}\centering 
\includegraphics[trim={4cm 3.0cm 4.0cm 0cm},clip,width=3.25cm,height=4.05cm,scale=0.66]{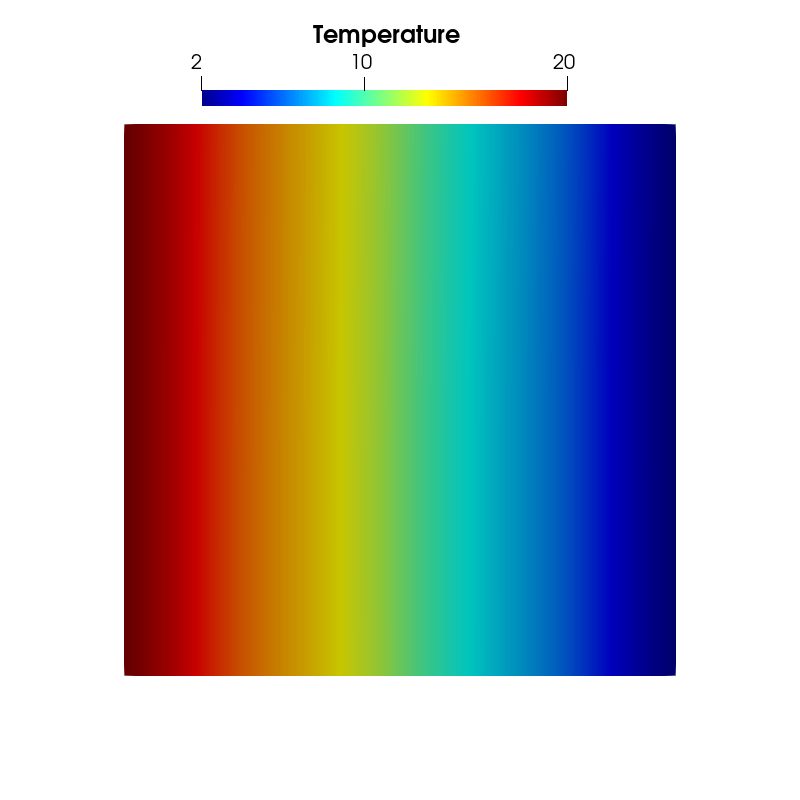} \\
\quad \tiny{(C.8)}
\end{minipage}
\begin{minipage}[b]{0.327\textwidth}\centering  
\includegraphics[trim={4cm 3.0cm 4.0cm 0cm},clip,width=3.25cm,height=4.05cm,scale=0.66]{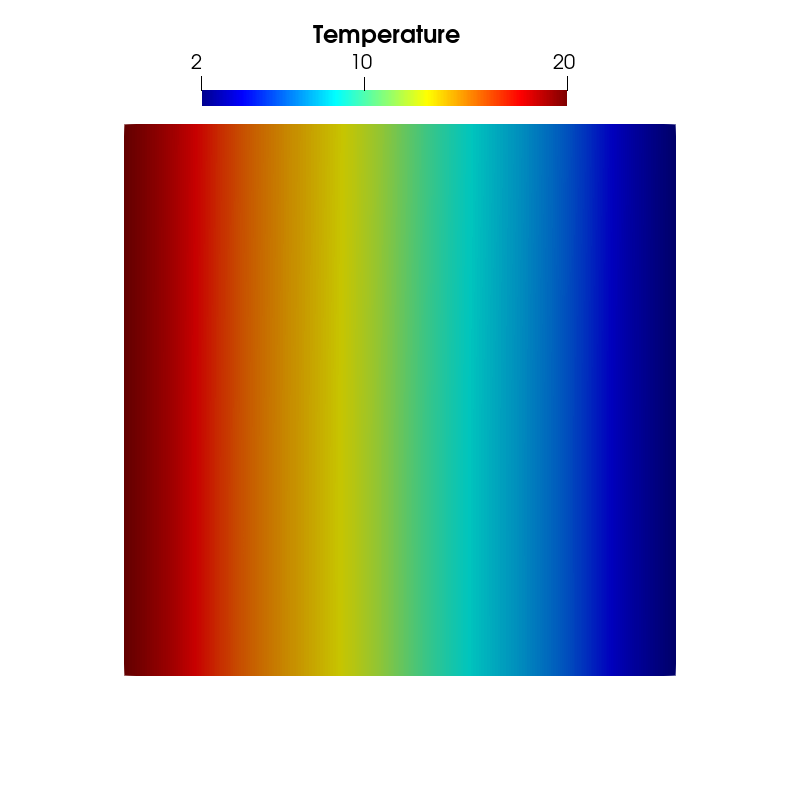} \\
\quad \tiny{(C.9)}
\end{minipage}

\caption{Example 2. Streamlines of the velocity field obtained for $s=3.0$ (C.1), $s=3.5$ (C.2), and $s=4.0$ (C.3), the contour lines of the pressure for $s=3.0$ (C.4), $s=3.5$ (C.5), and  $s=4.0$ (C.6), and temperature variable for $s=3.0$ (C.7), $s=3.5$ (C.8), $s=4.0$ (C.9).}
\label{fig:test_03}
\end{figure}

%

\bibliographystyle{siamplain}
\bibliography{biblio_sin_url}
\end{document}